\documentclass[11pt,reqno]{amsart}

\usepackage{amsmath, amsthm, amssymb}
\usepackage{amsfonts}
\usepackage{mathrsfs}

\usepackage{hyperref}

\usepackage[ansinew]{inputenc}
\usepackage[dvips]{epsfig}
\usepackage{graphicx}
\usepackage[english]{babel}
\usepackage[thinc]{esdiff}

\usepackage{csquotes} 
\usepackage[capitalize,nameinlink]{cleveref}

\usepackage{enumerate}
\usepackage{enumitem}
\usepackage{multirow}

\usepackage{xcolor}

\usepackage{thmtools}
\theoremstyle{plain}
\declaretheorem[title=Theorem, parent=section]{theorem}
\declaretheorem[title=Lemma,sibling=theorem]{lemma}
\declaretheorem[title=Proposition,sibling=theorem]{proposition}
\declaretheorem[title=Corollary,sibling=theorem]{corollary}

\theoremstyle{definition}
\declaretheorem[title=Definition,sibling=theorem]{definition}
\declaretheorem[title=Remark,sibling=theorem]{remark}

\numberwithin{equation}{section}

\usepackage[backgroundcolor=white, bordercolor=blue,
linecolor=blue]{todonotes}

\usepackage{dsfont}
\usepackage{bbm}

\newcommand{\N}{\mathbb{N}}
\newcommand{\R}{\mathbb{R}}

\newcommand{\cC}{\mathcal{C}}
\newcommand{\cP}{\mathcal{P}}

\newcommand{\gC}{\mathscr{C}}

\newcommand{\eps}{\varepsilon}

\newcommand{\bH}{\textbf{H}}

\newcommand{\average}{{\mathchoice {\kern1ex\vcenter{\hrule height.4pt
width 6pt depth0pt} \kern-9.7pt} {\kern1ex\vcenter{\hrule
height.4pt width 4.3pt depth0pt} \kern-7pt} {} {} }}

\usepackage{scalerel}
\usepackage{stackengine}
\newcommand{\fint}{\ThisStyle{\ensurestackMath{%
      \stackinset{c}{0\LMpt}{c}{0\LMpt}{\SavedStyle-}{\SavedStyle\phantom{\int}}}%
    \setbox0=\hbox{$\SavedStyle\int\,$}\kern-\wd0}\int}

\begin{document}

\allowdisplaybreaks

 \title[Dimension of the singular set 
 in the parabolic obstacle problem]{Dimension of the singular set \\ in the parabolic obstacle problem}

\author{Alejandro Mart\'inez}

\author{Xavier Ros-Oton}

\address{Departament de Matem\`atiques i Inform\`atica, Universitat de Barcelona, Gran Via de les Corts Catalanes 585, 08007 Barcelona, Spain}
\email{amartinezsanchez@ub.edu}

\address{ICREA, Pg. Llu\'is Companys 23, 08010 Barcelona, Spain \& Universitat de Barcelona, Departament de Matem\`atiques i Inform\`atica, Gran Via de les Corts Catalanes 585, 08007 Barcelona, Spain \& Centre de Recerca Matem\`atica, Barcelona, Spain}
\email{xros@icrea.cat}



\begin{abstract}
In this paper we study the singular set in the parabolic obstacle problem for general obstacles $\varphi \in C^{2,1}$, with $\Delta \varphi < 0$. We prove that the singular set has parabolic Hausdorff dimension at most $n-1$. Prior to our result, this was only known when $\Delta \varphi \equiv -1$. Our approach combines a truncated parabolic frequency formula and monotonicity estimates with an iterative argument showing that the frequency is saturated for all values of the truncation parameter between $2$ and $3$.
\end{abstract}

\allowdisplaybreaks

\maketitle
\section{Introduction}  

We consider the parabolic obstacle problem
\begin{align}
\label{eq:raw_parabolic_obstacle_problem}
\begin{cases}
\begin{alignedat}{2}
\partial_t v - \Delta v \;&=\; 0 
&\quad& \text{in } \{ v > \varphi \} \cap \Omega \times (0,T) \\[0.3em]
v \;&\geq\; \varphi 
&& \text{in } \Omega \times (0,T) \\[0.3em]
\partial_t v - \Delta v \;&\ge\; 0 
&& \text{in } \Omega \times (0,T), 
\end{alignedat}
\end{cases}
\end{align}
with boundary condition $v = g$ and initial condition $v(x,0)=\varphi(x)$, where $\varphi \in C^{2,1}(\overline{\Omega})$ is the obstacle.

This problem can be formulated as a variational inequality and it arises (among other contexts) in the optimal stopping problem for the Brownian motion; see \cite{Eva13}. In particular, it arises also in Mathematical Finance, in the Black-Scholes model for pricing of American options \cite{LS09}. We refer to the books \cite{DL76, KS80, Rod87, Fri88, PSU12} for further applications and motivations of (parabolic) obstacle problems.

Considering $u := v - \varphi$ and localizing to an interior cylinder, the parabolic obstacle problem becomes
\begin{equation}
\label{eq:parabolic_obstacle_problem}
\left\{
\begin{alignedat}{2}
\partial_t u - \Delta u \;&=\; -f(x)\,\chi_{\{u>0\}}
&\quad& \text{in } B_1\times(-1,1) \\[0.3em]
u \;&\ge\; 0
&& \text{in } B_1\times(-1,1) \\[0.3em]
\partial_t u \;&\ge\; 0
&& \text{in } B_1\times(-1,1) \\[0.3em]
\partial_t u \;&>\; 0
&& \text{in } \{u>0\},
\end{alignedat}
\right.
\end{equation}

where $f := -\Delta \varphi$ is a smooth nonnegative function. When $f\equiv 1$, this is also known as the (one-phase) Stefan problem.

One of the central challenges in obstacle-type problems is to understand the regularity of the free boundary. For \eqref{eq:parabolic_obstacle_problem}, little was known until the groundbreaking paper \cite{Caf77}. In this work, Caffarelli studied the case $f\equiv 1$ and proved that the free boundary is locally $C^\infty$ outside a closed set of singular points $\Sigma$, at which the contact set $ \{ u = 0\}$ has zero density.

More results on the singular set have been obtained in subsequent works. Using techniques adapted from the elliptic obstacle-type problem, it was shown in \cite{Bla06, LM15} that, at each fixed time $t$, the set of singular points $\Sigma_t$ is locally contained in an $(n-1)$-dimensional $C^1$ manifold. Other related works have explored structural and regularity aspects of parabolic obstacle problems: Hausdorff dimension of the free boundaries in related problems \cite{Wei99}; regularity for changing-sign solutions \cite{CPS04}; asymptotic behavior of solutions \cite{CSV20} and global regularity of the free boundary \cite{FRT26} among others. 

Despite these results, one question remained unanswered: can one prove that the singular set is $(n-1)$-dimensional in $(x,t)\in\R^n\times\R$? This was finally established in the recent paper \cite{FRS24} in case $f\equiv 1$. Moreover, they proved also that the singular set can be covered by $C^\infty$ $(n-1)$-dimensional manifolds, up to a smaller set $\Gamma$ of dimension $n-2$. Building on the work \cite{FRS24}, the recent paper \cite{Col25} proved that such set $\Gamma$ is actually parabolically countably $(n-2)$-rectifiable.

Our main goal in this paper is to extend the dimension bound obtained in \cite{FRS24}, previously known only for the constant case $f \equiv 1$, to arbitrary positive Lipschitz functions $f$. As explained below, this generalization is not trivial and requires not only controlling the errors in several monotonicity formulas, but also an iterative argument that is different from \cite{FRS24}.

\begin{theorem}
    \label{thm:main_theorem}
    Let $\Omega \subset \R^n$ be any open set, let $u \in L^\infty(\Omega \times (-T, T))$ solve the parabolic obstacle problem \eqref{eq:parabolic_obstacle_problem} with $f\in C^{0, 1}(\bar{\Omega} )$ and $f > 0$. Let $\Sigma \subset \Omega \times (-T, T)$ be the set of singular points. Then,
    $$
    \dim_\text{par}(\Sigma) \leq n-1,
    $$
    where $\dim_\text{par}$ denotes the parabolic Hausdorff dimension.
\end{theorem}
For a definition of the parabolic Hausdorff dimension see Section \ref{sec:notation_and_formulas}.

\subsection{The singular set}
\label{ssec:the_singular_set}

Under the assumption $f > 0$, points on the free boundary $\partial\{u >0\}$ can be classified into two categories: we call $(x_0, t_0)$ a regular point if and only if
$$
\lim_{r\downarrow 0} \frac{u(x_0 + rx,t_0 + r^2t)}{r^2} = \frac{f(x_0)}{2} \max \{\textbf{e} \cdot x, 0 \}^2
$$
for some $\textbf{e} \in \mathbb{S}^{n-1}$. We call $(x_0, t_0)$ a singular point if and only if
$$
\lim_{r\downarrow 0} \frac{u(x_0 + rx,t_0 + r^2t)}{r^2} = p_{2, x_0, t_0} (x)
$$
where $p_{2, x_0, t_0}(x)$ is a quadratic polynomial of the form $\frac{1}{2}x^TAx$, with $A$ being a symmetric nonnegative definite matrix such that $\text{tr}(A) = f(x_0)$. In both cases the convergence is in $C^{1, \alpha}_x \cap C^{0, \alpha}_t$ locally in $\R^{n+1}$. We will also denote by $p_2(x) := p_{2, 0, 0}(x)$.

This classification can be found for example in \cite[Theorem 5.44]{FR22} for the elliptic obstacle problem and \cite[Theorem 1.1]{AK23} for the fully nonlinear parabolic obstacle problem. 

It is known that the free boundary is $C^\infty$ in a neighborhood of regular points; however, its behavior can be significantly more complicated near singular points. In \cite{Sch77} it was shown that, in the elliptic obstacle problem, if $f$ is not strictly positive, the free boundary can be an arbitrary closed set; in particular, a set of dimension $n$. Moreover, under the hypothesis $f > 0$, the free boundary can be any closed subset of a hyperplane; in particular, a set of dimension $n-1$. For a fixed time similar constructions can be done in the parabolic obstacle problem.
  
\subsection{Ideas of the proof}

If $(x_0, t_0)$ is a singular point, we have the following expansion
\begin{equation}
    \label{eq:intro_basic_expansion}
    u(x_0 + x, t_0 + t) = p_{2, x_0, t_0}(x) + o(|x|^{2} + |t|),
\end{equation}

where $p_{2, x_0, t_0}(x)$ is the polynomial introduced previously. Since a singular point $(x_0,t_0)$ is characterized by the existence of such a polynomial $p_{2, x_0, t_0}$, a natural way to classify singular points is by using the dimension of $\{p_{2, x_0, t_0} = 0\}$. For $m \in \{0, 1, 2, ..., n-1 \}$ we define
$$
\Sigma_m := \left\{ (x_0, t_0) \in \Sigma: \dim\{p_{2, x_0, t_0} = 0\} = m\right\}, 
$$
hence the singular set decomposes as $\Sigma = \bigcup_{m=0}^{n-1} \Sigma_m$.

 Several ingredients in our proof refine techniques introduced in \cite{FRS24}. We develop a sharper version of \eqref{eq:intro_basic_expansion} in the top stratum $\Sigma_{n-1}$. To do this, we will use a truncated version of the frequency function
 $$
\phi^\gamma(r, w) := \frac{2r^2\int_{\{t = -r^2\}} |\nabla(u - p_2)|^2 G dx + \gamma r^{2\gamma} }{\int_{\{t = -r^2\}} (u - p_2)^2 G dx +  r^{2\gamma}}
 $$
 where $\gamma$ is a parameter to be chosen. Morally, this truncated version returns the value $\gamma$ if the frequency is very high. In this case, we will often refer to it as "frequency saturated". The authors of \cite{FRS24} proved that when $f \equiv 1$ the frequency function is almost-monotone for any value of $\gamma \in (2, \infty)$; however, this is not true in our case since, in our initial estimates, the frequency function is almost-monotone only for $\gamma \in \left(2, \frac{9}{4} \right)$. This constitutes the first step of our argument. Using this almost-monotonicity, we can improve the error term of \eqref{eq:intro_basic_expansion}, resulting in an error of $o(|x|^\frac{9}{4} + |t|^\frac{9}{8})$. With this improvement in the error term, we have sharper estimates and an almost-monotonicity formula can be obtained for $\gamma \in \left(2, \frac{39}{16} \right)$. Iterating this argument, we obtain an almost-monotonicity formula for $\gamma \in \left(2, 3 - \eps \right)$.

Once this extended almost-monotonicity is established, the proof of Theorem \ref{thm:main_theorem} proceeds as follows.
If $(x_0,t_0) \in \Sigma_m$ with $m \le n-2$, the expansion \eqref{eq:intro_basic_expansion} cannot be improved. Nevertheless, a barrier argument similar to \cite{FRS20,FRS24} yields $\dim_{\mathrm{par}}(\Sigma_m) \le n-2$.
If instead $(x_0,t_0) \in \Sigma_{n-1}$, the improved almost-monotonicity allows us to strengthen the expansion to
 \begin{equation}
\label{eq:target_improved_expansion_formula}
 u(x_0 + x, t_0 + t) = p_{2, x_0, t_0}(x) + o(|x|^{3 - \eps} + |t|^{(3 - \eps)/2}),
 \end{equation}
for any $\eps >0$.
This refined expansion implies $\dim_{\mathrm{par}}(\Sigma_{n-1}) \le n-1$, and combining the two cases completes the proof.

A key observation allowing us to assume only that $f$ is Lipschitz is the following: if a subharmonic function is Lipschitz in all but one direction, then it is $C^{0,\alpha}$ in that remaining direction for every $\alpha \in (0,1)$. This is proved in detail in \cref{prop:cc_argument}. We expect Lipschitz regularity of $f$ to be the minimal regularity assumption under which \eqref{eq:target_improved_expansion_formula} holds.

\subsection{Organization of the paper} 

In Section \ref{sec:notation_and_formulas} we introduce the main notation and definitions and collect several preliminary results used throughout the paper. In Section \ref{sec:weiss_mon_form} we introduce several almost-monotonicity formulas and some crucial estimates. In Section \ref{sec:2nd_blowup_top_stratum} we analyze the second blow-up in the top stratum, i.e. at the critical points belonging to $\Sigma_{n-1}$. In Section \ref{sec:iterative_argument} we provide a refined version of the almost-monotonicity formulas (assuming extra hypotheses) and prove that the frequency formula is saturated for $\gamma < 3$. Finally, in Section \ref{sec:proof_of_main_thm}, we combine the results from the previous sections to estimate the size of the singular set and prove \cref{thm:main_theorem}.

\subsection{Acknowledgments} AM and XR were supported by the European Research Council under the Grant Agreement No. 101123223 (SSNSD), and by AEI project PID2024-156429NB-I00 (Spain).

\smallskip
\section{Notation and Preliminary Results}
\label{sec:notation_and_formulas}

In this section, we introduce the notation and collect some preliminary results that will be used throughout the paper. Since this article generalizes the first main theorem in \cite{FRS24}, we adopt similar notation.

\subsection{Notation and definitions}
We start by introducing some notation.
\subsection*{Function $f$} Throughout the paper, $f$ is a fixed positive Lipschitz function and we denote by
$$
L := [f]_{C^{0,1}(B_1)}
$$
its Lipschitz constant.
\subsection*{Operators} We define the following operators 
\begin{align*}
\textbf{H} := (\Delta - \partial_t) \qquad \text{and} \qquad Z:= (x \cdot \nabla + 2t\partial_t). 
\end{align*}
We can rewrite parabolic obstacle problem \eqref{eq:parabolic_obstacle_problem} using them
\begin{align}
\label{eq:stefan_problem}
\begin{cases}
    \textbf{H}u = f \chi_{\{u > 0\}} & \\
    u \geq 0  &\qquad \text{in } B_1 \times (-1, 1) \\
    \partial_t u > 0 \quad \text{in } \{u > 0\}
\end{cases}
\end{align}
\subsection*{Bilinear form} For $r \in (0, 1]$ and functions $g, h: \R^n \times (-1, 0) \rightarrow \R$ we define the following bilinear form
\begin{align*}
    \langle g, h\rangle_r := \int_{\R^n} (g h)(x, -r^2)G(x, -r^2) dx
\end{align*}

where $G(x, t)$ is the reversed heat kernel in $\R^n \times (-\infty, 0)$, defined as follows:
\begin{align*}
    G(x, t) := \frac{1}{(4\pi t)^{n/2}} e^{|x|^2/(4t)}.
\end{align*}
In our applications $g$ and $h$ will be at least continuous although the bilinear form can be defined for more general functions.

\subsection*{Restricted bilinear forms}
Since in our applications we will distinguish between the contact set $\{u = 0\}$ and the positivity region $\{u > 0\}$, we also define the following restricted bilinear forms
$$
\langle g, h\rangle_r^{0} := \int_{\{u(\cdot, -r^2) = 0\}} (gh)(x, -r^2)\, G(x, -r^2)\, dx \quad \text{and} \quad \langle g, h\rangle_r^{+} := \int_{\{u(\cdot, -r^2) > 0\}} (gh)(x, -r^2)\, G(x, -r^2)\, dx,
$$
where both sets are understood as subsets of $\R^n$. Since they partition $\R^n$, we have $\langle g, h \rangle_r = \langle g, h \rangle_r^{0} + \langle g, h \rangle_r^{+}$. From now on we will refer to them as $\{u = 0\}$ and $\{u > 0\}$, the relevant time slice will be clear from the context.

\subsection*{Functionals D and H} We also define the following functionals, which correspond to dimensionless quantities, and are the parabolic version of the functionals that were introduced in the elliptic problem at \cite{FS17} and adapted to the Stefan problem in \cite{FRS24}.
\begin{align*}
D(r, w) := 2r^2 \langle \nabla w, \nabla w \rangle_r \qquad \text{and} \qquad H(r, w) := \langle w, w \rangle_r
\end{align*}
where $w:\R^n \times (-1, 0) \rightarrow \R$ will be, in our applications, at least $C^{1, 1}_x \cap C^{0, 1}_t(B_1 \times (-1, 1))$.

\subsection*{Frequency functions} We introduce the following functions
\begin{align*}
\phi(r, w) := \frac{D(r, w)}{H(r, w)}, \qquad \phi^\gamma(r, w) := \frac{D(r, w) + \gamma r^{2\gamma}}{H(r, w) + r^{2 \gamma}}
\end{align*}
The function $\phi(r, w)$ is known as the parabolic version of Almgren frequency function, and $\phi^\gamma(r,w)$ is a truncated version introduced in \cite{FRS24}.

In the remainder of this section, we introduce several auxiliary notions that will be used throughout the analysis and the geometric setting adapted to the parabolic scaling.

\subsection*{Parabolic rescaling} Following the standard notation for parabolic PDE problems given a function $w: \R^n \times (-1, 0) \to \R$ we define
$$
w_r(x, t):= w(rx, r^2t)    
$$ for $r > 0$.

Note that $Z(w_r)(x, t) = (Zw)(rx, r^2t)$ and $\mathbf{H}(w_r)(x,t) = r^2 (\mathbf{H}w)(rx, r^2t)$.

\subsection*{Spatial cut-off} We fix a smooth spatial cut-off function $\zeta:\R^n \to \R$ verifying $\zeta \in C^\infty_c(B_{1/2})$ and $\zeta \equiv 1$ in $B_{1/4}$.

\subsection*{Family of polynomials} Let $\mathcal{P}$ denote the set of nonnegative $2$-homogeneous polynomials in the $x$ variable satisfying $\Delta p = f(0)$. Since they do not depend on $t$, every $p \in \mathcal{P}$ solves $\bH p = f(0)$ in $\R^n \times \R$.

\subsection*{Blow-up polynomial} If $(x_0,t_0)$ is a singular point, we denote by 
$p_{2,x_0, t_0}$ its blow-up polynomial, i.e.
$$
\lim_{r\downarrow 0} \frac{u(x_0 + rx,t_0 + r^2t)}{r^2} = p_{2, x_0, t_0} (x),
$$
see \cref{ssec:the_singular_set}. In particular, we denote $p_2 := p_{2, 0, 0}$, and note that $p_2 \in \mathcal{P}$.

Throughout the paper, we will frequently work on parabolic cylinders, which reflect the natural scaling of the problem. We denote them as follows:
\subsection*{Parabolic cylinders} Given $r > 0$ we define the parabolic cylinder $\gC_r$ as
$$
\gC_r := B_r \times(-r^2, 0).
$$
It will also be suitable to measure distance in space-time with respect to the parabolic geometry associated with the heat operator. For this purpose, we introduce the parabolic distance:
\subsection*{Parabolic distance and parabolic Hausdorff dimension} Given $(x, t), (y, s) \in \R^n \times \R$, we define the parabolic distance by
$$
\text{dist}_\text{par} \left( (x, t), (y, s)\right) := |x - y| + |t - s|^{1/2}.
$$
The parabolic Hausdorff dimension is the Hausdorff dimension associated with the parabolic distance; see \cite[Chapter 2]{Mat15}.

\subsection*{Spatial projection} We denote by $\pi_x: \R^n \times \R \to \R^n$ the canonical projection $\pi_x(x, t) := x$.

\subsection{Preliminary results}
\label{sec:preliminary_results}
We present some formulas that will be useful in the proofs of Section \ref{sec:weiss_mon_form} and several auxiliary results that will be used throughout the paper.

\begin{lemma} 
    \label{lem:formulas_for_H&D}
    For $w \in C^1$, the following holds:
    \begin{enumerate}
        \item $H'(r, w) = 2 r^{-1}\langle w, Zw\rangle_r$ 
        \item $D(r, w) = \langle w, Zw \rangle_r -2 r^2\langle w, \textbf{H} w\rangle_r$
        \item $D'(r, w) = 2r^{-1} \langle Zw, Zw\rangle_r - 4r \langle Zw, \textbf{H} w\rangle_r$
    \end{enumerate}
\end{lemma}

\begin{proof}
    The formulas are the rescaled version of \cite[Lemma 4.3, Lemma 4.5]{FRS24}.
\end{proof}

\smallskip
\begin{lemma}[\cite{FRS24}, Lemma 4.6]
\label{lem:frequency_error_formula}
\begin{align*}
\diff{}{r}\phi^\gamma(r, w) \geq \frac{2}{r}\frac{\left(\langle Zw, Zw \rangle_r \langle w, w \rangle_r - \langle w, Zw \rangle_r^2\right) + \left(2r^2\langle w, \textbf{H}w \rangle_r\right)^2 + E^\gamma(r, w) }{\left( H(r, w) + r^{2\gamma}\right)^2}
\end{align*}
where
\begin{align*}
    E^\gamma(r,w) := 2r^2 \langle w, \bH w \rangle_r \left(D(r, w) + \gamma r^{2\gamma} \right) - 2r^2 \langle Zw, \textbf{H}w \rangle_r \left( H(r, w) + r^{2\gamma} \right).
\end{align*}
\end{lemma}

We have the following result for the singular set, that is a consequence of Theorem 1.9 in \cite{LM15}.
\begin{theorem}
    \label{thm:cover_of_singular_projection}
    Let $u\in C_x^{1, 1} \cap C_t^{0, 1}\left(B_1 \times (-1, 1)\right)$ solve \eqref{eq:stefan_problem}. Denote by $\Sigma$ the set of singular points. Then $\pi_x(\Sigma) \subset B_1$ can be locally covered by a $(n-1)$-dimensional $C^1$ manifold.
\end{theorem}

The following lemma was proved for $f \equiv 1$ in \cite{Caf77} and adapted to the fully nonlinear parabolic obstacle problem in \cite[Theorem 5.1]{AK23}.
\smallskip
\begin{lemma}
    \label{lem:constant_simplification_bound}
    Let $u : B_1\times(-1,1) \to [0,\infty)$ be a bounded solution of \eqref{eq:stefan_problem} with $(0,0) \in \partial\{u=0\}$. Then

$$\sup_{B_{1/2}\times(-1/2,0)} |D^2u| + |\partial_t u| \leq C,$$

where $C$ depends only on $n$, $\|u(\cdot,0)\|_{L^\infty(B_1)}$ and $\|f\|_{L^\infty(B_1)}$.
\end{lemma}

Note that, since we assume $f$ to be Lipschitz we have 
$$\|f\|_{L^\infty(B_1)} \le f(0) + L,$$
hence the dependence of $\|f\|_{L^\infty(B_1)}$ in the constant of \cref{lem:constant_simplification_bound} can be replaced by a dependence on $f(0)$ and $L$.

We will also use the half-Harnack inequality for subcaloric functions (see \cite[Theorem 4.16]{Wa92}):
\begin{lemma}
    \label{lem:half-harnack}
    Let $w: \gC_1 \rightarrow \R$ satisfy $\bH w \geq0$. Then
    $$
    \sup_{\gC_{1/2}}w \leq C \left( \int_{\gC_1} (w_+)^\eps\right)^{1/\eps},
    $$
    for some $\eps \in (0, 1)$ and $C$ depending only on $n$.
\end{lemma}

The parabolic version of the Calder\'on-Zygmund estimate \cite{Jon64} reads as follows.

\smallskip
\begin{lemma}
    \label{lem:parabolic_CZ_estimate}
    Let $w: \gC_1 \to \R$ be any solution of $\bH w= g$ in $\gC_1$, with $g \in L^1$. Then
    $$
    \sup_{\theta > 0} \theta \left| \{|D^2w| + |\partial_t w| > \theta \} \cap \gC_{1/2}\right| \leq C \left( \|g\|_{L^1(\gC_1)} + \|w\|_{L^1(\gC_1)}\right),
    $$
    for some constant $C$ depending only on $n$.
\end{lemma}

We also need the following estimate for subharmonic functions, which we prove here.
\smallskip
\begin{proposition}
    \label{prop:cc_argument}
    Let $u: B_1 \to \R$ be a subharmonic function such that $\| \partial_{x_i} u \|_{L^\infty(B_1)} \leq C$ for all $i \in \{1, 2, \dots, n-1\}$. Then for all $\alpha \in (0, 1)$, there exists a constant $C = C( \alpha, n)$ such that
    $$
    [u]_{C^{0, \alpha}(B_{1/2})} \leq C \left(\| \nabla 'u\|_{L^\infty(B_1)} +\| u\|_{L^\infty(B_1)} \right), 
    $$

    where $\nabla'$ denotes the gradient w.r.t. the first $n-1$ variables.
\end{proposition}

To prove it, we need two lemmas: the first one is an intermediate step introducing a dependence of $\delta [u]_{C^{0, \alpha}(B_1)}$ that will be proved using a contradiction-compactness argument; the second one is a technical lemma to obtain the final bound. 

\smallskip
\begin{lemma}
    \label{lem:cc_argument} Fix $\alpha \in (0, 1)$ and $\delta > 0$. Let $u \in C^{0, \alpha}(B_1)$ be any subharmonic function such that $\| \nabla' u \|_{L^\infty(B_1)} \leq C$, where $\nabla'$ denotes the gradient w.r.t. the first $n-1$ variables. Then there exists a constant $C_\delta = C(\delta, \alpha, n)$ such that
    $$
    [u]_{C^{0, \alpha}(B_{1/2})} \leq \delta [u]_{C^{0, \alpha}(B_1)} + C_\delta \left(\| \nabla 'u\|_{L^\infty(B_1)} +\| u\|_{L^\infty(B_1)} \right).
    $$
\end{lemma}

\begin{proof}
    Assume the claim is false. Then there exists a $\delta_0 > 0$ and a sequence of subharmonic functions $\{u_k \}_{k \in \N}$ such that
    \begin{equation*}
        \label{eq:porp_cc_negated_claim}
        [u_k]_{C^{0, \alpha}(B_{1/2})} > \delta_0 [u_k]_{C^{0, \alpha}(B_1)} + k \left(\| \nabla 'u_k\|_{L^\infty(B_1)} +\| u_k\|_{L^\infty(B_1)} \right).
    \end{equation*}
For each $k$ take $x_k, y_k \in B_{1/2}$ such that
$$
    \frac{|u_k(x_k) -u_k(y_k)|}{|x_k - y_k|^\alpha} \geq \frac{1}{2}  [u_k]_{C^{0, \alpha}(B_{1/2})}.
$$
Note $\rho_k := |x_k - y_k| \to 0$ as $k \to \infty$ since
$$
     \frac{1}{2}[u_k]_{C^{0, \alpha}(B_{1/2})}  \leq \frac{|u_k(x_k) -u_k(y_k)|}{|\rho_k|^\alpha} \leq  \frac{2\| u_k\|_{L^\infty(B_1)}}{|\rho_k|^\alpha} < \frac{2[u_k]_{C^{0, \alpha}(B_{1/2})}}{k\rho_k^\alpha}
$$
which implies $\rho_k^\alpha < 4/k$. We define the blow-up sequence by
$$
\tilde{u}_k(x) = \frac{u_k(x_k +\rho_k x) - u_k(x_k)}{\rho_k^\alpha [u_k]_{C^{0, \alpha}(B_1)}}.
$$
For any $R < \frac{1}{2\rho_k}$ note $[\tilde{u}_k]_{C^{0, \alpha}(B_R)} \leq 1$ then by Arzel\`a-Ascoli theorem $\tilde{u}_k \to \tilde{u}$ up to subsequences in $C^{0, \alpha}_\text{loc}(\R^n)$. Also note that 
$$
    \Delta \tilde{u}_k(x) = \rho_k^{2 - \alpha} \frac{\Delta u_k(x_k + \rho_k x)}{[u_k]_{C^{0, \alpha}(B_1)}} \geq 0
$$
and
$$
\partial_{x_i} \tilde{u}_k(x) = \rho_k^{1-\alpha} \frac{\partial_{x_i} u_k(x_k + \rho_k x)}{[u_k]_{C^{0, \alpha}(B_1)}} \leq \rho_k^{1-\alpha} \frac{\partial_{x_i} u_k(x_k + \rho_k x)}{k\| \nabla' u_k\|_{L^\infty(B_1)}}.
$$
Therefore the limit function $\tilde{u}$ is also subharmonic and has $\partial_{x_i} \tilde{u} = 0$ for all $i \in \{1, 2, \dots, n-1 \}$, which implies $\tilde{u}(x) = \phi(x_n)$ for some $\phi : \R \to \R$ convex. A convex function on $\mathbb{R}$ that has strictly sub-linear growth must be a constant. However if we take
$$
    \xi_k = \frac{y_k - x_k}{\rho_k} \in \mathbb{S}^{n-1},
$$
(note that, up to subsequences, $\xi_k \to \xi \in \mathbb{S}^{n-1}$) and evaluate
$$
|\tilde{u}_k(\xi_k)| = \frac{|u_k(y_k) - u_k(x_k)|}{\rho_k^\alpha [u_k]_{C^{0, \alpha}(B_1)}} \geq \frac{1}{2} \frac{[u_k]_{C^{0, \alpha}(B_{1/2})}}{[u_k]_{C^{0, \alpha}(B_1)}} > \frac{\delta_0}{2}
$$
we conclude $|\tilde{u}(\xi)| > \delta_0 /2$ which rules out the case of $\tilde{u}$ being constant since we also have $\tilde{u}(0) = 0$.
\end{proof}

\smallskip
\begin{lemma}[{\cite[Lemma 2.27]{FR22}}]
    \label{lem:elliptic_PDE_2.27}
    Let $\eta \in \R$ and $\beta > 0$. Let $S$ be a non-negative function on the class of open convex subsets of $B_1$ and suppose that $S$ is sub-additive. That is, if $A, A_1, \dots, A_N$ are open convex subsets of $B_1$ with $A \subset \bigcup_{j=1}^NA_j$, then $S(A) \leq \sum_{j = 1}^N S(A_j)$. Then, there exists $\delta > 0$ small (depending only on $n$ and $\eta$) such that, if
    $$
    \rho^\eta S(B_{\rho/2}(x_0)) \leq \delta \rho^\eta S(B_\rho(x_0)) + \beta \quad \text{for all } B_\rho(x_0) \subset B_1,
    $$
    then 
    $$
    S(B_{1/2}) \leq C\beta
    $$
    for some $C$ depending only on $n$ and $\eta$.
\end{lemma}
 We can now give the proof of the proposition:
\begin{proof}[Proof of \cref{prop:cc_argument}]
    It follows from \cref{lem:cc_argument} and \cref{lem:elliptic_PDE_2.27} taking $S(E) = \| u \|_{C^{0, \alpha}(E)}$, $\eta = \alpha$ and $\beta = C \left(\| \nabla 'u\|_{L^\infty(B_1)} +\| u\|_{L^\infty(B_1)} \right)$.
\end{proof}

\smallskip
\section{Weiss and frequency formulae}
\label{sec:weiss_mon_form}
We would like to apply the results of this section to $u - p_2$. However, since we are considering only solutions defined in $B_1 \times (-1, 1)$ and our functionals $D$ and $H$ are defined in the whole space we have to use the cut-off function introduced in Section \ref{sec:notation_and_formulas}. This fact generates an error term which can be controlled thanks to the following lemma.

\smallskip
\begin{lemma}[{\cite[Lemma 5.2]{FRS24}}]
\label{lem:exp_error}
Let $v \in C^{1, 1}_x \cap C^{0, 1}_t(\gC_1)$ and $\zeta$ a fixed cut-off function. Then for all $r\in (0, 1/2)$ we have 
$$
|\langle \zeta v, \bH(\zeta v) \rangle_r - \langle \zeta v, \zeta \bH v \rangle_r| + |\langle Z(\zeta v), \bH(\zeta v) \rangle_r - \langle \zeta Z v, \zeta \bH v \rangle_r| \leq CM_v^2e^{-\frac{1}{(8r)^2}}
$$
where
$$
M_v := \sup_{B_{1/2} \times (-1/2, 0)} |v| + |\nabla v| + |\partial_t v| + |\Delta v|,
$$
and $C$ depends only on $\zeta$.
\end{lemma}

Two key ingredients in the results from \cite{FRS24} were $\langle u - p, \bH (u-p) \rangle_r \geq 0$ and $\langle Z(u - p), \bH (u-p) \rangle_r \geq 0$ where $u$ is a solution of \eqref{eq:stefan_problem} with $f \equiv 1$ problem and $p \in \cP$. However if $f \not\equiv 1$ the previous estimates are no longer true, it may happen $\langle u - p, \bH (u-p) \rangle_r < 0$ or $\langle Z(u - p), \bH (u-p) \rangle_r < 0$. For that reason we distinguish two cases: in the positivity set $\{ u > 0\}$ they need to be refined and in the contact set $\{ u =0 \}$ the estimates from \cite{FRS24} hold (it is only necessary to take into account the exponentially small error introduced by the cutoff).
\smallskip
\begin{lemma} 
    \label{lem:cubic_estimates}
    Let $u: B_1 \times (-1, 1) \rightarrow [0, \infty)$ be a bounded solution of \eqref{eq:stefan_problem} and $(0, 0)$ a singular point. Given $p \in \cP$ set $w:= (u - p)\zeta$. Then, there exists $r_0 > 0$ such that, in the region $\{u > 0 \}$, it holds
    \begin{equation*}
        | \langle w, \bH w \rangle_r^+| \leq Cr^3 \qquad \text{and} \qquad | \langle Zw, \bH w \rangle_r^+| \leq Cr^3 \qquad \text{for all } r \in (0, r_0),
    \end{equation*}
    where $C$ is a constant depending only on $L$, $f(0)$, $n$, $\|p\|_{L^\infty(B_1)}$ and $\|u(\cdot, 0) \|_{L^\infty(B_1)}$.
\end{lemma}
\begin{proof}
     Since $u$ is a solution of \eqref{eq:stefan_problem} in the region $\{u > 0\}$ it holds $\bH(u-p) = f - f(0)$. Since $f$ is Lipschitz, we have $|(\bH (u-p))_r(x)| = |f(rx) - f(0)| \leq Lr|x|.$
    Moreover, since $u \in C^{1,1}_x \cap C^{0,1}_t(\gC_1)$ and $(0,0)$ is a singular point, we have $(u-p)(0,0) = 0$ and $\nabla(u-p)(0,0) = 0$. Hence, $|(u-p)_r(x, t)| \leq Cr^2|x|^2$, $|(\nabla(u-p))_r(x, t)| \leq Cr|x|$ and $|(\partial_t(u-p))_r(x,t)| \leq C$ where
    the constants depend only on $n$, $L$, $f(0)$, $\|p\|_{L^\infty(B_1)}$ and $\|u(\cdot, 0) \|_{L^\infty(B_1)}$ thanks to \cref{lem:constant_simplification_bound}.

    For the first estimate, using \cref{lem:exp_error} we can replace $\langle w, \bH w\rangle_r^+$ by $\langle \zeta(u-p), \zeta \bH(u-p)\rangle_r^+$ up to an exponentially small error, and using $0 \leq \zeta \leq 1$, we get
    $$
    \left|\langle \zeta(u - p), \zeta\bH (u - p)\rangle_r^+\right| 
    \leq C\int_{\{u_r(\cdot,-1)>0\}} r^2|x|^2 \cdot r |x| \, G(x, -1) dx \leq C\int_{\R^n} r^2|x|^2 \cdot r |x| \, G(x, -1) dx \leq Cr^3 ,
    $$
    
    and the exponential error can be absorbed into the cubic term, note that using \cref{lem:constant_simplification_bound} the constant $M_v$ from \cref{lem:exp_error} can be controlled by $n$, $L$, $f(0)$, $\|p\|_{L^\infty(B_1)}$ and $\|u(\cdot, 0) \|_{L^\infty(B_1)}$. In the last inequality we have also used $
    \int |x|^3G(x, -1) \leq C$, where this constant also depends on $n$.

    For the second estimate note that, on $\{t = -1\}$,
    $$
    Z(u-p)_r = \left( Z (u-p) \right)_r = rx \cdot (\nabla(u-p))_r - 2r^2 (\partial_t(u-p))_r,
    $$
    so that $|Z(u-p)_r(x,-1)| \leq Cr^2(|x|^2 + 1)$. Arguing as above,
    $$
    \left|\langle \zeta Z(u - p), \zeta\bH (u - p)\rangle_r^+\right| 
    \leq C\int_{\{u_r(\cdot,-1)>0\}} r^2(|x|^2+1) \cdot r|x| \, G(x, -1)\, dx \leq Cr^3 ,
    $$
    and the claim follows.
\end{proof} 
\smallskip
\begin{lemma} 
    \label{lem:contact_set_FRS_5.1}
    Let $u: B_1 \times (-1, 1) \rightarrow [0, \infty)$ be a bounded solution of \eqref{eq:stefan_problem} and $(0, 0)$ a singular point. Given $p \in \cP$, set $w:= (u - p)\zeta$. Then, there exists $r_0 > 0$ such that, inside of the contact set $\{u = 0\}$, it holds 
    \begin{equation*}
           \langle w, \bH w \rangle_r^0 \geq -Ce^{-\frac{1}{r^2}} \qquad \text{and} \qquad \left|\langle 2w - Zw, \bH w \rangle_r^0\right| \leq Ce^{-\frac{1}{r^2}} \qquad \text{for all } r \in (0, r_0),
    \end{equation*}
    where $C > 0$ depends only on $L$, $f(0)$, $n$, $\|p\|_{L^\infty(B_1)}$ and $\|u(\cdot, 0) \|_{L^\infty(B_1)}$.
\end{lemma}
\begin{proof}
     In the contact set $\{ u = 0\}$ we have $w = -p \zeta$. For the first estimate note that  $\bH (u-p) = -\bH p = -f(0)$ and $(u - p) = -p$. Hence
    \begin{equation*}
        \langle(u-p) , \bH (u-p) \rangle_r^0 = f(0)\int_{\{u(\cdot,-r^2)=0\}}(pG)(x, -r^2) dx \geq 0.
    \end{equation*}
    For the second estimate note that $Z(u-p) = Z(-p) = -2p$ (since we have $\nabla u = \partial_t u = 0$ a.e. and $p$ is a $2$-homogeneous polynomial), thus, 
    \begin{equation*}
        \langle 2(u-p) - Z(u-p), \bH (u-p) \rangle_r^0 = \langle 2p - Zp, \bH p \rangle_r^0 = 0.
    \end{equation*}
    Combining the previous estimates with the exponentially small error introduced by the cutoff function, \cref{lem:exp_error}, the claim follows.
\end{proof}

\begin{remark}
    The exponentially small errors that appear when introducing the cutoff $\zeta$ can be bounded using \cref{lem:exp_error}, as has been done in the proof of \cref{lem:cubic_estimates}. Those error terms can always be absorbed by polynomial error terms.
\end{remark}
We introduce now our first almost-monotonicity result:

\begin{lemma}[Weiss'-type almost monotonicity formula]
\label{lem:weiss_type_formula}
Let $u: B_1 \times (-1, 1) \rightarrow [0, \infty)$ be a bounded solution of \eqref{eq:stefan_problem} and $(0, 0)$ a singular point. Given $p \in \cP$, set $w:= (u - p)\zeta$ and let
$$
W(r, w) := \frac{1}{r^4}(D(r, w) - 2 H(r, w)).
$$
Then, there exists $r_0 > 0$ such that:
\begin{enumerate}
    \item For all $r \in (0, r_0)$
    $$
    \diff{ \ }{r}W(r, w) \geq -C,
    $$
    where $C$ depends only on  $L$, $f(0)$, $n$, $\|p\|_{L^\infty(B_1)}$ and $\|u(\cdot, 0) \|_{L^\infty(B_1)}$.
    \item $W(0^+, w) = 0$.
    
    \item For all $r \in (0, r_0)$
    $$
    -Cr^5 \leq  D(r, w) - 2H(r, w) \leq Cr^4
    $$
    where $C > 0$ depends only on $L$, $f(0)$, $n$, $\|p\|_{L^\infty(B_1)}$ and $\|u(\cdot, 0) \|_{L^\infty(B_1)}$.
\end{enumerate}
\end{lemma}
\begin{proof}
    \mbox{}\\[0.6em] 
    \noindent \textit{(a)} Using \cref{lem:formulas_for_H&D} and \cref{lem:cubic_estimates}, the first one for computing the derivatives $D'(r, w)$ and $H'(r, w)$ and the second one for controlling the terms $\langle w, \bH w \rangle_r$ and $\langle Zw, \bH w \rangle_r$, we obtain
    \begin{equation}
        \label{eq:weiss_formula}
        \begin{aligned}
            \diff{ \ }{r}W(r, w) &= \frac{D'(r, w) - 2H'(r, w)}{r^4} - 4\frac{D(r, w) - 2H(r, w)}{r^5}  \\
            & = \frac{2}{r^5} \left( \langle Zw, Z w \rangle_r - 2r^2 \langle Zw, \bH w\rangle_r - 4\langle w, Zw \rangle_r + 4r^2 \langle w, \bH w \rangle_r + 4 \langle w, w \rangle_r\right)  \\
            & =\frac{2}{r^5}(\langle Zw - 2w, Zw -2w \rangle_r + 4r^2 \langle w, \bH w \rangle_r -2r^2\langle Zw, \bH w\rangle_r)\\
            & = \frac{2}{r^5}(\langle Zw - 2w, Zw -2w \rangle_r + 2r^2 \langle 2w -Zw, \bH w \rangle_r ). 
        \end{aligned}
    \end{equation}
    We distinguish two cases: In the region $\{ u > 0\}$ the cubic estimates from \cref{lem:cubic_estimates} hold and then $\langle 2w -Zw, \bH w \rangle_r^+ \geq -Cr^3$. In the contact set $\{ u = 0\}$, by \cref{lem:contact_set_FRS_5.1}, we have $\langle 2w -Zw, \bH w \rangle_r^0 \geq -C e^{-\frac{1}{r^2}}\geq -Cr^3$. Adding the two contributions, we have $2r^2\langle 2w -Zw, \bH w \rangle_r \geq -Cr^5$ and, substituting back in \eqref{eq:weiss_formula}, we get
    \begin{equation*}
        \begin{aligned}
            \diff{ \ }{r}W(r, w) \geq \frac{2}{r^5}\langle Zw - 2w, Zw -2w \rangle_r -C   \geq -C. 
        \end{aligned}
    \end{equation*}
    since the exponential error can be absorbed in the polynomial terms. 
    \par\smallskip
    \noindent \textit{(b)} To compute the limit note that, since $(0, 0)$ is a singular point, we have $(u-p)_r = (p_2 - p)_r + o(r^2)$, so
    $$
    W(0^+, w) = \lim_{r \downarrow 0 } W(r, w) = \lim_{r \downarrow 0 } W(1, r^{-2}w_r) = \lim_{r \downarrow 0 } W(1, p_2 - p).   
    $$
    Integrating by parts and using $\Delta p_2 = \Delta p$, $\nabla G = -\frac{x}{2}G$ and also $x \cdot \nabla(p_2 - p) = 2(p_2 - p)$ since $p_2$ and $p$ are $2-$homogeneous, we get
    \begin{align*}
        D(1, p_2 - p) &= 2\int_{\{t = -1\}} \nabla(p_2 - p) \cdot \nabla(p_2 - p) G\\
        &= -2\int_{\{t = -1\}}(p_2 - p)G\Delta(p_2-p) - 2\int_{\{t = -1\}} (p_2 - p) \nabla(p_2 - p) \cdot \nabla G  \\ &= \int_{\{t = -1\}} (p_2 - p) x \cdot \nabla(p_2 - p)G = 2 H(1, p_2 - p).
    \end{align*}
    \par\smallskip
    \noindent \textit{(c)} The lower bound is a direct consequence of (a) and (b). Integrating \eqref{eq:weiss_formula} from $0$ to $r$ we obtain
    $$
    W(r, w) = \frac{1}{r^4} \left(D(r, w) -2H(r, w) \right) \geq -Cr.
    $$
    For the upper bound integrating the formula from a) we get
    $$
    W(r, w) = \frac{1}{r^4} \left(D(r, w) -2H(r, w) \right) \leq C(r_0 - r) + W(r_0, w) \leq C,
    $$
    and the claim follows.
\end{proof} 

\smallskip
\begin{lemma} 
    \label{lem:freq_direct_bound}
    Let $u: B_1 \times (-1, 1) \rightarrow [0, \infty)$ be a bounded solution of \eqref{eq:stefan_problem} and $(0, 0)$ a singular point. Given $p \in \cP$ and $\gamma \in \left(2, \frac{5}{2} \right)$, set $w:= (u - p)\zeta$. Then there exists $r_0 > 0$ such that for all $r \in (0, r_0)$
    \begin{equation}
        \label{eq:frequency_formula_bound}
        0 \leq \phi^\gamma(r, w) - 2 \leq Cr^{4 -2\gamma},
    \end{equation}
    where $C$ is a constant depending only on $L$, $f(0)$, $n$, $\|p\|_{L^\infty(B_1)}$ and $\|u(\cdot, 0) \|_{L^\infty(B_1)}$. 
\end{lemma}
\begin{proof}
    We start by the lower bound, using \cref{lem:weiss_type_formula} c) we get
    \begin{equation*}
        \begin{aligned}
            \phi^\gamma(r, w) - 2 & = \frac{D(r, w) + \gamma r^{2\gamma} - 2\left(H(r, w) + r^{2\gamma} \right)}{H(r, w) + r^{2\gamma}} \\
            & = \frac{D(r, w) - 2H(r, w) + (\gamma - 2)r^{2\gamma}}{H(r, w) + r^{2\gamma}} \\
            & \geq \frac{-Cr^5 + (\gamma - 2)r^{2\gamma}}{H(r, w) + r^{2\gamma}} \\
            & = \frac{ r^{2\gamma}\left(-Cr^{5-2\gamma} + \gamma - 2 \right)}{H(r, w) + r^{2\gamma}} \geq 0,
        \end{aligned}
    \end{equation*}
    where, the last inequality, we have used the fact that, since $\gamma \in \left(2, \frac{5}{2} \right)$, the term $(\gamma - 2)$ is positive, then, $-Cr^{5-2\gamma} + \gamma - 2 \geq 0$ provided that $r$ is small enough.

    For the upper bound we use again \cref{lem:weiss_type_formula} c), note 
    \begin{equation*}
        \begin{aligned}
            \phi^\gamma(r, w) - 2 & = \frac{D(r, w) + \gamma r^{2\gamma} - 2\left(H(r, w) + r^{2\gamma} \right)}{H(r, w) + r^{2\gamma}} \\
            & = \frac{D(r, w) - 2H(r, w) + (\gamma - 2)r^{2\gamma}}{H(r, w) + r^{2\gamma}} \\
            & \leq \frac{Cr^4 + (\gamma - 2)r^{2\gamma}}{H(r, w) + r^{2\gamma}} \leq C r^{4-2\gamma},
        \end{aligned}
    \end{equation*}
    and the claim follows.
\end{proof}

\begin{lemma}
    \label{lem:frequency_formula_monotonicity}
    Let $u: B_1 \times (-1, 1) \rightarrow [0, \infty)$ be a bounded solution of \eqref{eq:stefan_problem} and $(0, 0)$ a singular point. Given $p \in \cP$, set $w:= (u - p)\zeta$. Then, for all $\gamma \in \left(2, \frac{9}{4} \right)$ there exists $r_0 > 0$ such that for all $r \in (0, r_0)$ it holds
    \begin{equation}
        \label{eq:frequency_formula_monotonicity}
        \diff{}{r} \phi^\gamma(r, w) \geq \frac{2}{r} \left(\frac{2r^2 \langle w, \bH w\rangle_r}{H(r, w) + r^{2\gamma}}\right)^2 -Cr^{\eps - 1}
    \end{equation}
    for some $\eps \in (0, 1)$, where $C$ is a constant depending only on $L$, $f(0)$, $n$, $\|p\|_{L^\infty(B_1)}$ and $\|u(\cdot, 0) \|_{L^\infty(B_1)}$.
\end{lemma}
\begin{proof}
    Using the direct computation of the derivative of the frequency formula (\cref{lem:frequency_error_formula}) we get
    \begin{align*}
\diff{}{r}\phi^\gamma(r, w) \geq \frac{2}{r}\left(\frac{\left(2r^2\langle w, \textbf{H}w \rangle_r\right)^2 + E^\gamma(r, w) }{\left( H(r, w) + r^{2\gamma}\right)^2}\right).
\end{align*}
We can bound the error term by
    \begin{equation}
        \label{eq:freq_formula_tmp_step}
        \begin{aligned}
            \frac{E^\gamma(r, w)}{r\left(H(r, w) + r^{2\gamma}\right)^2} &= \frac{2r \langle w, \bH w \rangle_r \left(D(r, w) + \gamma r^{2\gamma} \right) - 2r \langle Zw, \bH w \rangle_r \left( H(r, w) + r^{2\gamma} \right)}{\left(H(r, w) + r^{2\gamma}\right)^2} \\
        &=\frac{2r \langle w, \bH w \rangle_r \phi^\gamma(r, w) - 2r\langle Zw, \bH w \rangle_r}{H(r, w) + r^{2\gamma}} \\
        &= \frac{2r \langle w, \bH w \rangle_r \left(\phi^\gamma(r, w) -2\right) + 2r\langle2w - Zw, \bH w \rangle_r}{H(r, w) + r^{2\gamma}}. 
        \end{aligned}
    \end{equation}
    where in the last equality we have added and subtracted $4r\langle w, \bH w \rangle_r$. By \cref{lem:freq_direct_bound}
    we have $0 \leq \phi^\gamma(r, w) -2 \leq Cr^{4-2\gamma}$. In the region $\{u > 0\}$ the cubic estimates from \cref{lem:cubic_estimates} hold and we have $2r \langle w, \bH w \rangle_r^+ \left(\phi^\gamma(r, w) -2\right) \geq -Cr^{8-2\gamma}$ and $2r\langle2w - Zw, \bH w \rangle_r^+ \geq -Cr^4$. In the contact set $\{u = 0\}$, by \cref{lem:contact_set_FRS_5.1} we have the bound $2r \langle w, \bH w \rangle_r^0 \left(\phi^\gamma(r, w) -2\right) + 2r\langle2w - Zw, \bH w \rangle_r^0 \geq -Cre^{-\frac{1}{r^2}} \geq -Cr^4$. In either case, substituting back in \eqref{eq:freq_formula_tmp_step}, we obtain
    $$
    \frac{E^\gamma(r, w)}{r\left(H(r, w) + r^{2\gamma}\right)^2} \geq -C(r^{8-4\gamma} + r^{4-2\gamma}) \geq -Cr^{8-4\gamma},
    $$
    since $\gamma > 2$. Setting $\eps = 9 - 4\gamma$ (which is positive since $\gamma < 9/4$) completes the proof.
\end{proof}

The previous results allow us to make the following definitions.
\begin{definition}
    \label{def:lambdas}
    Let $u: B_1 \times (-1, 1) \rightarrow [0, \infty)$ be a bounded solution of \eqref{eq:stefan_problem} and $(0, 0)$ a singular point. Given $p \in \cP$ and $\gamma \in (2, 3)$ we define
    $$
    \lambda := \lim_{r \downarrow0} \phi^\gamma \left(r, (u-p)\zeta \right)
    \qquad \text{and} \qquad
    \lambda^* := \lim_{r \downarrow0} \phi^\gamma \left(r, (u-p_2)\zeta \right),
    $$
    whenever these limits exist. By \cref{lem:freq_direct_bound} and \cref{lem:frequency_formula_monotonicity}, this is the case for every $\gamma \in \left(2, \frac{9}{4}\right)$. In \cref{sec:iterative_argument} we show that $\lambda^*$ also exists for larger values of $\gamma$.
\end{definition}

We introduce the following Monneau-Type monotonicity formula.
\smallskip
\begin{lemma}
    \label{lem:monotonicity_H}
    Let $u: B_1 \times (-1, 1) \rightarrow [0, \infty)$ be a bounded solution of \eqref{eq:stefan_problem} and $(0, 0)$ a singular point. Given $p \in \cP$, set $w:= (u - p)\zeta$. Then
    $$
    \diff{}{r} \left( \frac{H(r, w)}{r^4}\right) \geq -C
    $$

    where $C$ depends only on $L$, $f(0)$, $n$, $\|p\|_{L^\infty(B_1)}$ and $\|u(\cdot, 0) \|_{L^\infty(B_1)}$.
\end{lemma}
\begin{proof}
    By the formulas from \cref{lem:formulas_for_H&D},
    $$
    \diff{}{r} \left( \frac{H(r, w)}{r^4}\right) = \frac{rH'(r, w) - 4H(r, w)}{r^5} = \frac{2\left(D(r, w)-2H(r, w)\right) + 4r^2\langle w, \bH w\rangle_r}{r^5} \geq -C.
    $$
    In the last inequality we have used \cref{lem:weiss_type_formula} c) and $\langle w, \bH w\rangle_r \geq -Cr^3$ from \cref{lem:cubic_estimates} and \cref{lem:contact_set_FRS_5.1}.
\end{proof}

\smallskip
\begin{lemma}
    \label{lem:FRS_5.6}
    Let $u: B_1 \times (-1, 1) \rightarrow [0, \infty)$ be a bounded solution of \eqref{eq:stefan_problem} and $(0, 0)$ a singular point. Given $p \in \cP$, set $w:= (u - p)\zeta$ and let $\gamma \in (2, 3)$. Assume
    $$
    \diff{}{r} \phi^\gamma(r, w) \geq \frac{2}{r} \left(\frac{2r^2 \langle w, \bH w\rangle_r}{H(r, w) + r^{2\gamma}}\right)^2 -C_0r^{\eps - 1} \qquad \text{and} \qquad \frac{2r^2\langle w, \bH w \rangle_r}{H(r, w) + r^{2\gamma}} \geq -C_0r^\eps
    $$
    
    for some $\eps \in (0, 1)$ and $C_0$. Recall that $\lambda = \phi^\gamma(0^+, w)$ (\cref{def:lambdas}) and assume that there exists $\delta > 0$ and $R \in (0, 1)$ such that $\phi^\gamma(r, w) \leq \lambda + \frac{\delta}{4}$ for all $r \in (0, R)$. Then 

    \begin{equation}
    \label{eq:main_comparison_FRS_5.6}
        c \left(\frac{R}{r} \right)^{2\lambda} \leq \frac{H(R, w) + R^{2\gamma}}{H(r, w) + r^{2\gamma}} \leq C_\delta \left( \frac{R}{r} \right)^{2\lambda + \delta}    
    \end{equation}
    where $c, C_\delta > 0$ depend only on  $\eps$ and $C_0$; $C_\delta$ also depends on $\delta$.
    
    Moreover, the lower bound in \eqref{eq:main_comparison_FRS_5.6} holds for all $0 < r < R$ with $R$ in the range where the first two assumptions hold, without the assumption $\phi^\gamma(r, w) \leq \lambda + \frac{\delta}{4}$.
\end{lemma}
\begin{proof}
    Define
    $$
    F(r) := \frac{2r^2 \langle w, \bH w\rangle_r}{H(r, w) + r^{2\gamma}}
    $$

    and note that, by assumption,
    $$
    \diff{}{r} \phi^\gamma(r, w) + Cr^{\eps - 1} \geq \frac{2}{r} (F(r))^2.
    $$

    Using the formula for $H'(r, w)$ from \cref{lem:formulas_for_H&D} we get
    \begin{equation}
    \label{eq:equality_derivative_H}
    \frac{\diff{}{r}(H(r, w)+r^{2\gamma})}{H(r, w)+r^{2\gamma}} =  \frac{ 2\left(D(r, w) + \gamma r^{2\gamma} + 2r^2 \langle w, \bH w \rangle_r \right)}{r\left(H(r, w)+r^{2\gamma} \right)} =\frac{2}{r} \left(\phi^\gamma(r, w) + F(r) \right).
    \end{equation}
    Integrating between $r$ and $R$ we obtain
    \begin{equation}
        \label{eq:log_H+r/H+r}
    \log \frac{H(R, w)+R^{2\gamma}}{H(r, w)+r^{2\gamma}} \leq \int_r^R \frac{2\lambda +\delta/2 + 2F(s)}{s} ds \leq \left(2\lambda + \frac{\delta}{2} \right) \log(R/r) + \int_r^R \frac{2}{s}F(s) .
    \end{equation}
    We need to control the term with $F(s)$, to do so we can use Cauchy-Schwarz
    \begin{align}
        \label{eq:control_of_F}
        \left| \int_r^R \frac{1}{s} F(s)ds\right| &\leq \left( \int_r^R\frac{1}{s}(F(s))^2ds\right)^{1/2} \left( \int_r^R \frac{1}{s}ds\right)^{1/2} \nonumber \\ 
        &\leq \left( \int_r^R \frac{1}{2} \left( \diff{}{s}\phi^\gamma(s, w) + Cs^{\eps - 1}\right)ds\right)^{1/2} \left( \log(R/r) \right)^{1/2} \\
        &\leq \left( \frac{\lambda}{2} + \frac{\delta}{8} +\frac{C}{2\eps}(R^\eps - r^\eps)\right)^{1/2} \left( \log(R/r) \right)^{1/2} \leq C \left( \log(R/r) \right)^{1/2}. \nonumber
    \end{align}
    Substituting \eqref{eq:control_of_F} in \eqref{eq:log_H+r/H+r} we get the upper bound
    $$
    \log \frac{H(R, w)+R^{2\gamma}}{H(r, w)+r^{2\gamma}} \leq \left( 2\lambda + \frac{\delta}{2} \right) \log(R/r) + C\left(\log(R/r) \right)^{1/2} \leq \left( 2\lambda + \delta \right) \log(R/r) + C_\delta, 
    $$
    where the term $C \left( \log(R/r) \right)^{1/2}$ was absorbed. 
    
    For the lower bound note that, for $r$ small enough, we have
    $$
    \phi^\gamma(r, w) \geq \lambda - \frac{Cr^\eps}{\eps}.
    $$
    Using $F(r) \geq -Cr^\eps$ and substituting in \eqref{eq:equality_derivative_H} we get the lower bound
    $$
    \frac{\diff{}{r}(H(r, w)+r^{2\gamma})}{H(r, w)+r^{2\gamma}} \geq \frac{2}{r} \left( \lambda - C\frac{\eps + 1}{\eps} r^\eps\right).
    $$
    Integrating from $r$ to $R$ we get
    \begin{align*}
        \label{eq:log_H+r/H+r}
    \log \frac{H(R, w)+R^{2\gamma}}{H(r, w)+r^{2\gamma}} 
    &\geq \int_r^R \frac{2}{s}\left(\lambda -C\frac{\eps+1}{\eps}s^\eps \right) ds \\
    &\geq 2\lambda\log(R/r) - \int_r^R 2C\frac{\eps + 1}{\eps}s^{\eps-1} ds
    \geq 2\lambda \log(R/r) -C_\eps,
    \end{align*}
    and the claim follows.
\end{proof}

\begin{remark}
    \cref{lem:FRS_5.6} is the parabolic analogue of \cite[Lemma 4.1]{FRS20}, and corresponds to \cite[Lemma 5.6(b)]{FRS24}, where it is stated for $f \equiv 1$ and the proof is mentioned as a minor modification of \cite[Lemma 4.1]{FRS20}. In \cite{FRS24} the error terms in the frequency formula are exponentially small, while for Lipschitz $f$ they are only polynomial, as in the hypotheses of \cite[Lemma 4.1]{FRS20}, with $\eps$ playing the role of $\kappa$. The proof follows that of \cite[Lemma 4.1]{FRS20}, replacing $r^{2-n}\int_{B_r} w \Delta w$ by our parabolic equivalent $\langle w, \bH w \rangle_r$ and the elliptic identity for the derivative of $\log \left(H(r, w) + r^{2\gamma} \right)$ by \eqref{eq:equality_derivative_H}. The upper bound in \eqref{eq:main_comparison_FRS_5.6} corresponds to \cite[Lemma 4.1(a)]{FRS20} and the lower bound, without loss in the exponent, to \cite[Lemma 4.1(b)]{FRS20}. Also note that, since $ \lambda \leq \gamma < 3$, the constants do not depend on an a priori upper bound for $\phi^\gamma$, as they do in \cite{FRS20}.
\end{remark}

\begin{remark}
    \label{rmk:hypotheses_FRS_5.6}
    For $\gamma \in \left(2, \frac{9}{4}\right)$ both assumptions of \cref{lem:FRS_5.6} hold with $\eps = 9 - 4\gamma$. The first one is \cref{lem:frequency_formula_monotonicity}. For the second, by \cref{lem:cubic_estimates} and \cref{lem:contact_set_FRS_5.1} we have $\langle w, \bH w\rangle_r \geq -Cr^3$, hence
    $$
    \frac{2r^2\langle w, \bH w \rangle_r}{H(r, w) + r^{2\gamma}} \geq -Cr^{5 - 2\gamma} \geq -Cr^{9 - 4\gamma},
    $$
    since $5 - 2\gamma > 9 - 4\gamma$ for $\gamma > 2$. In particular, $C_0$ depends only on $L$, $f(0)$, $n$, $\|p\|_{L^\infty(B_1)}$ and $\|u(\cdot, 0)\|_{L^\infty(B_1)}$.
\end{remark}

\smallskip
\begin{corollary} 
    \label{cor:lambda_leq_gamma}
    Let $u: B_1 \times (-1, 1) \rightarrow [0, \infty)$ be a bounded solution of \eqref{eq:stefan_problem} and $(0, 0)$ a singular point. Let $\gamma \in (2,3)$ be such that the assumptions of \cref{lem:FRS_5.6} hold and $\lambda$ be given by \cref{def:lambdas}. Then $\lambda \leq \gamma$.
\end{corollary}
\begin{proof}
    By the lower bound in \eqref{eq:main_comparison_FRS_5.6} with $R$ fixed, $r^{2\gamma} \leq H(r, w) + r^{2\gamma} \leq C r^{2\lambda}$ for all $r$ small enough, hence $\lambda \leq \gamma$.
\end{proof}

\smallskip
\begin{corollary}
    \label{cor:H_assymptotic_bhvr}
    Let $u: B_1 \times (-1, 1) \rightarrow [0, \infty)$ be a bounded solution of \eqref{eq:stefan_problem} and $(0, 0)$ a singular point. Set $w:= (u - p)\zeta$, let $\gamma \in (2,3)$ be such that the assumptions of \cref{lem:FRS_5.6} hold  and $\lambda$ be given by \cref{def:lambdas}. Assume $\lambda < \gamma$, then 
    $$
    \lim_{r\downarrow0}\frac{r^{2\gamma}}{H(r, w)} = 0.
    $$
    Moreover, for any $K$ there exists $C_K > 1$ such that for $r$ small enough
    $$
    \frac{1}{C_K} \leq \frac{H(r, w)}{H(\theta r, w)} \leq C_K \qquad \theta \in [K^{-1}, K].
    $$
\end{corollary}

\begin{proof}
    Taking $\delta > 0$ fixed, such that $2\lambda + \delta < 2\gamma$ and using \eqref{eq:main_comparison_FRS_5.6} with $R$ fixed we obtain
    $$
    H(r, w) + r^{2\gamma} \geq Cr^{2\lambda + \delta}.
    $$
    Dividing by $r^{2\gamma}$ and taking limit when $r\downarrow0$ we get
    $$
    \lim_{r\downarrow0}\frac{H(r, w)}{r^{2\gamma}} + 1 \geq C\lim_{r\downarrow0}r^{2\lambda + \delta - 2\gamma} = +\infty,
    $$
    which implies
    $$
    \lim_{r\downarrow0}\frac{r^{2\gamma}}{H(r, w)} = 0.
    $$
    This proves that $H(r, w)$ is the dominating term in the expression $H(r, w) + r^{2\gamma}$. Moreover, taking $K^{-1} \leq \theta \leq 1$ (the case $1 \leq \theta \leq K$  is analogous) we obtain
    $$
    \lim_{r\downarrow 0} \frac{H(r, w) + r^{2\gamma}}{H(\theta r, w) + (\theta r)^{2\gamma}} = \lim_{r\downarrow 0} \frac{H(r, w)}{H(\theta r, w)},
    $$
    and the last claim of the corollary follows from \eqref{eq:main_comparison_FRS_5.6}.
\end{proof}

\smallskip
\begin{corollary}
    \label{cor:H_and_D_bound}
    Let $u: B_1 \times (-1, 1) \rightarrow [0, \infty)$ be a bounded solution of \eqref{eq:stefan_problem} and $(0, 0)$ a singular point. Set $w:= (u - p_2)\zeta$, let $\gamma \in (2,3)$ be such that the assumptions of \cref{lem:FRS_5.6} hold and $\lambda^*$ be given by \cref{def:lambdas}. Then, for $r$ small enough,
    $$
    H(r, w) + D(r, w) \leq Cr^{2\lambda^*},
    $$
    where $C$ depends only on $\gamma$, $L$, $f(0)$, $n$, $\|u(\cdot, 0) \|_{L^\infty(B_1)}$ and on $\eps$ and $C_0$ from the assumptions of \cref{lem:FRS_5.6}. 
\end{corollary}
\begin{proof}
    Let $r_0$ and $\eps$ be such that the assumptions of \cref{lem:FRS_5.6} hold on $(0,r_0)$, and set $R := \frac{r_0}{2}$. By the lower bound in \eqref{eq:main_comparison_FRS_5.6},
    $$
        H(r, w) + r^{2\gamma} \leq \frac{H(R, w) + R^{2\gamma}}{c\,R^{2\lambda^*}}\, r^{2\lambda^*} \leq C\, r^{2\lambda^*} \qquad \text{for all } r \in (0,R).
    $$
    Since $H(R, w)$ is bounded in terms of $n$ and $\|u(\cdot,0)\|_{L^\infty(B_1)}$ the bound for $H(r,w)$ follows. For $D(r, w)$, the almost-monotonicity of $\phi^\gamma$ gives
    $$
    \phi^\gamma(r,w) \leq \phi^\gamma(R,w) + \frac{C}{\eps}R^\eps \leq \frac{D(R,w) + \gamma R^{2\gamma}}{R^{2\gamma}} + \frac{C}{\eps}R^\eps \leq C \qquad \text{for all } r \in (0,R),
    $$
    and hence
    $$
    D(r, w) \leq \phi^\gamma(r, w)\left( H(r, w) + r^{2\gamma}\right) \leq C r^{2\lambda^*},
    $$
    as wanted.
\end{proof}
\smallskip
\section{The 2nd Blow-up in the Top Stratum}
\label{sec:2nd_blowup_top_stratum}
In this section we will analyze the structure of the blow-ups of $u - p_2$ for singular points in $\Sigma_{n-1}$. We present estimates that hold in general, estimates that hold only in the top stratum and several lemmas which compare the growth of $H(r, w)$ and the $L^2$-norm of $w_r$. 

We start proving some estimates that will allow us to control different quantities by $\|w_r\|_{L^2}$ up to a cubic error term.
\smallskip
\begin{lemma}
    \label{lem:L2_controls}
    Let $u: B_1 \times (-1, 1) \rightarrow [0, \infty)$ be a bounded solution of \eqref{eq:stefan_problem} and $(0, 0)$ a singular point. Set $w:= (u - p_2)$. Then there exists $r_0 > 0$ such that for all $r \in (0, r_0)$,
    $$
    \|w_r\|_{L^\infty(\gC_1)} + \|\nabla w_r\|_{L^2(\gC_1)} +\| \partial_t w_r\|_{L^2(\gC_1)} \leq C(\|w_r\|_{L^2(\gC_2)} + r^3),
    $$
    where $C$ depends on $n$, $f(0)$ and $L$. 
\end{lemma}

To prove this we will need the following auxiliary lemma from \cite{FRS24},
\smallskip
\begin{lemma}[{\cite[Lemma 6.1]{FRS24}}]
    \label{lem:FRS_6.1}
    Let $u_i : \gC_2 \to \R$, $i = 1, 2$ solve
    $$
    \begin{cases}
        \bH u_i = \chi_{\{u_i > 0 \}}(1 + \eps_i(x, t)) & \text{in} \ \ \gC_2 \\
        u_i \geq 0 \\
        \partial_t u_i \geq 0,
    \end{cases}
    $$
    with $|\eps_i(x, t)| \leq \bar{\eps} < \frac{1}{100}$, and set $w:= u_1 - u_2$. Then
    \begin{equation}
        \label{eq:L2_controls_Linf}
        \|w\|_{L^\infty(\gC_1)} \leq C(\|w\|_{L^2(\gC_2)} + \bar{\eps})
    \end{equation}
    and
    \begin{equation}
        \label{eq:L2_controls_int}
       \left(\int_{\gC_1} |\nabla w|^2  +|w \bH w| + |\partial_tw|^2\right)^{1/2} \leq C(\|w\|_{L^2(\gC_2)} + \bar{\eps})
    \end{equation}
    where $C$ is a dimensional constant.
\end{lemma}

\begin{proof}[Proof of \cref{lem:L2_controls}]
We will use \cref{lem:FRS_6.1} setting $u_1 = \frac{u_r(x, t)}{r^2 f(0)}$ and $u_2 = \frac{p_2(rx)}{r^2 f(0)}$. We have to check that $u_1$ fulfills the hypotheses of the lemma. Note that
$$
\bH u_1= \chi_{\{u_1 > 0 \}} \left( \frac{f(rx)}{f(0)}\right)
$$

and also, since $f$ is Lipschitz, for $|x| \leq 2$ (this condition appears since we are applying \cref{lem:FRS_6.1} in $\gC_2$) we have $\left|\frac{f(rx)}{f(0)} - 1\right| \leq \frac{2L}{f(0)}r$, which fulfills the hypothesis with $\bar\eps = \frac{2L}{f(0)}r$ provided $r_0 < \frac{f(0)}{200L}$. The case of $u_2$ is clear because $\bH u_2 = 1$, hence it solves the PDE in weak sense since $\{p_2 = 0 \}$ has zero measure in $\R^n$. 

Therefore, by \eqref{eq:L2_controls_Linf} we obtain
$$
\frac{\|(u-p_2)_r\|_{L^\infty(\gC_1)}}{r^2f(0)} \leq C \left( \frac{\|(u - p_2)_r\|_{L^2(\gC_2)}}{r^2f(0)} + \frac{2L}{f(0)}r\right)
$$
and the bound for $\|w_r\|_{L^\infty(\gC_1)}$ follows. The $L^2$ bounds for $\|\nabla w_r\|_{L^2(\gC_1)}$ and $\|\partial_t w_r\|_{L^2(\gC_1)}$ are obtained in a similar way from \eqref{eq:L2_controls_int}.
\end{proof}

\smallskip
\begin{lemma}
    \label{lem:L_infty_bounds}
    Let $u: B_1 \times (-1, 1) \rightarrow [0, \infty)$ be a bounded solution of \eqref{eq:stefan_problem}, $(0, 0) \in \Sigma_{n-1}$. Set $w := (u - p_2)$, and without loss of generality assume $p_2 = \frac{f(0)}{2} x_n^2$. Then for every $\alpha \in (0, 1)$, there exists $r_0 > 0$ such that for all $r \in (0, r_0)$ it holds
    $$
    \|\partial_t w_r\|_{L^\infty(\gC_1)} +  \|\nabla' w_r\|_{L^\infty(\gC_1)} + \| w_r \|_{C^{0, \alpha}_x(\gC_1)} \leq C\left( \|w_r\|_{L^2(\gC_2)} + r^3\right),  
    $$
    where $C$ is a constant depending only on $n$, $f(0)$, $L$ and $\alpha$. Here $\nabla'$ denotes the gradient w.r.t. the first $n-1$ variables and $\| \cdot \|_{C^{0, \alpha}_x}$ denotes the H\"older norm only w.r.t. the spatial variables. 
\end{lemma}
\begin{proof}
    For $\textbf{e} \in \mathbb{S}^{n-1} \cap\{p_2 = 0\}$ define the incremental quotients for a function $g: \R^{n+1} \to \R$, $\textbf{e} \in \mathbb{S}^{n-1}$ and $h \in (0, 1)$ 
    $$
    \delta_{(\textbf{e}, h)} g := \frac{g(\cdot+\textbf{e} h, \cdot) - g(\cdot, \cdot)}{h} \quad \text{and} \quad \delta_{(t, h)} g:= \frac{g(\cdot, \cdot + h) - g(\cdot, \cdot)}{h}.
    $$
    We want to compute the incremental quotients of $w_r = (u - p_2)_r$. Note that $\delta_{(\textbf{e}, h)} p_2 = 0$ since
    
    $\textbf{e} \in \{p_2=0\}\cap\mathbb{S}^{n-1}$ and also $\delta_{(t, h)} p_2 = 0$ since $p_2$ does not depend on time variable.

    In the region $\{u > 0\}$ using that $f$ is Lipschitz we have
    $$
    \bH(\delta_{(\textbf{e}, h)}w_r) \leq r^2  \frac{f_r(x + \textbf{e}h) - f_r(x)}{h}  \leq Cr^3.
    $$
    In the region $\{u= 0\}$ we have
    $$
    \delta_{(\textbf{e}, h)}w_r = \delta_{(\textbf{e}, h)}u_r  = \frac{u_r(x + \textbf{e}h, t)}{h} \geq 0.
    $$
    Hence the function $\left( \delta_{(\textbf{e}, h)}w_r\right)_- +\frac{C}{2n}r^3|x|^2$ is subcaloric since 
    $$
    \bH\left( ( \delta_{(\textbf{e}, h)}w_r)_- +\frac{C}{2n}r^3|x|^2 \right) = \bH \left( ( \delta_{(\textbf{e}, h)}w_r)_-\right) +Cr^3 \geq 0.
    $$
    Therefore, since $u \in C^{1, 1}_x \cap C^{0, 1}_t$ we have 
    $$(\delta_{(\textbf{e}, h)}w_r)_- +\frac{C}{2n}r^3|x|^2 \longrightarrow (\partial_{\textbf{e}}w_r)_- +\frac{C}{2n}r^3|x|^2  \quad \text{a. e. as } h \to 0
    $$ 
    and the limit function is also subcaloric. We can repeat all the bounds in direction $-\textbf{e}$ instead of $\textbf{e}$, since $\delta_{(-\textbf{e}, h)} w_r \to -\partial_\textbf{e} w_r$ and $(-\partial_\textbf{e}w_r)_- = (\partial_\textbf{e}w_r)_+$ we obtain that $|\partial_\textbf{e}w_r | + \frac{C}{2n}r^3|x|^2$ is a subcaloric function. 
    
    Similarly for the time variable, in the region $\{ u > 0\}$ we have
    $$
    \bH \left( -\delta_{(t, -h)} w_r\right) = r^2 \frac{f_r(x)\chi_r(x, t - h) - f_r(x)}{h} \leq 0.
    $$

    Inside $\{ u = 0\}$, we have $-\delta_{(t, - h)} w_r \geq 0$.
    Therefore $\left(-\delta_{(t, - h)} w_r\right)_-$ and its limit $(-\partial_t w_r)_-$ are subcaloric. Since $\partial_t w_r = \partial_t u_r \geq 0$ we have $\left(-\partial_t w_r\right)_- = \partial_t w_r = |\partial_t w_r|$, thus the absolute value is also subcaloric. 
    
    By \cref{lem:half-harnack} we get
    $$
    \sup_{\gC_{4/3}} \left( |\partial_t w_r| + |\partial_\textbf{e} w_r|+\frac{C}{2n}r^3|x|^2\right) \leq C \left( \int_{\gC_{3/2}}  \left( |\partial_t w_r| + |\partial_\textbf{e} w_r|+\frac{C}{2n}r^3|x|^2\right)^\eps \right)^{1/\eps},
    $$
    and using that $\|\cdot\|_{L^\eps}$ is a quasi norm
    $$
    \sup_{\gC_{4/3}} \left( |\partial_t w_r| + |\partial_\textbf{e} w_r|\right) \leq  C \left( \int_{\gC_{3/2}}  \left( |\partial_t w_r| + |\partial_\textbf{e} w_r|\right)^\eps\right)^{1/\eps} + Cr^3
    $$

    for some dimensional constant $C$. Since $L^{1, \infty} \subset L^\eps$ it holds:
    \begin{equation*}
            \left( \int_{\gC_{3/2}}  \left( |\partial_t w_r| + |\partial_\textbf{e} w_r|\right)^\eps\right)^{1/\eps} 
            \leq C \sup_{\theta > 0} \theta|\{|\partial_t w_r|+|\partial_\textbf{e} w_r| > \theta\} \cap \gC_{3/2}|. 
    \end{equation*}
    By Calderon-Zygmund (\cref{lem:parabolic_CZ_estimate}) we can control the $|\partial_t w_r|$ term by
    $$
    \sup_{\theta > 0} \theta|\{|\partial_t w_r| > \theta\} \cap \gC_{3/2}| \leq C\left( \|\bH w_r\|_{L^1(\gC_{5/3})} +\|w_r\|_{L^1(\gC_{5/3})}\right).
    $$
    and the term $|\partial_\textbf{e} w_r|$ can be controlled directly by \cref{lem:L2_controls},
    $$
    \| \partial_\textbf{e} w_r \|_{L^{1, \infty}(\gC_{3/2})} \leq \| \partial_\textbf{e} w_r \|_{L^{1}(\gC_{3/2})} \leq C \|\nabla w_r\|_{L^2(\gC_{3/2})} \leq C\left( \|w_r\|_{L^2(\gC_2)} + r^3\right).
    $$    
    Combining the previous estimates we obtain
    $$
    \sup_{\gC_{4/3}} \left( |\partial_t w_r| + |\partial_\textbf{e} w_r|\right) \leq  C\left( \|\bH w_r\|_{L^1(\gC_{5/3})} +\|w_r\|_{L^1(\gC_{5/3})} + r^3\right).
    $$
    Note $\bH w_r \leq Cr^3$, taking $\xi$ a smooth cut-off function which verifies $\xi \equiv 1$ in $\gC_{5/3}$ and vanishes outside $\gC_2$ and integrating by parts we can estimate
    $$
    \|\bH w_r -Cr^3\|_{L^1(\gC_{5/3})} \leq -\int_{\gC_2} \xi \left(\bH w_r -Cr^3 \right) = -\int_{\gC_2} (\Delta + \partial_t)\xi w_r + Cr^3 \leq C\left(\|w_r\|_{L^2(\gC_2)} +r^3\right).
    $$
    thus
    \begin{equation*}
        \label{eq:supremum_bound_FRS_6.5}
        \sup_{\gC_{4/3}} \left( |\partial_t w_r| + |\partial_\textbf{e} w_r|\right) \leq C\left(\|w_r\|_{L^2(\gC_2)} +r^3\right).    
    \end{equation*}
    Therefore we have the Lipschitz estimate for the first $n-1$ derivatives
    $$
    \|\nabla'w_r\|_{L^\infty(\gC_{4/3})} \leq C\left(\|w_r\|_{L^2(\gC_2)} +r^3\right).
    $$
     Since $\bH w_r \leq Cr^3$, the previous estimate gives $\Delta w_r = \bH w_r + \partial_t w_r \leq C\left(\|w_r\|_{L^2(\gC_2)} + r^3\right) =: M$, so that $-w_r + \frac{M}{2n}|x|^2$ is subharmonic. Applying a rescaled version of \cref{prop:cc_argument} to this function we obtain
    $$
    \|w_r\|_{C^{0, \alpha}_x(\gC_1)} \leq C\left(\|w_r\|_{L^2(\gC_2)} + r^3\right),
    $$
    and we are done.
\end{proof}

Next, we will prove some lemmas that will allow us to compare $\|w_r\|_{L^2(\gC_1)}$ and $H(r, w \zeta)^{1/2}$ at all scales. Note that we can only prove this result if $\lambda < \gamma$. In the case of $\lambda = \gamma$ we have no information about the rate of decrease of $H(r, w \zeta)^{1/2}$ which, a priori, could be arbitrarily fast. The following result is an auxiliary lemma that will be useful for proving an upper bound for $\|w_r\|_{L^2(\gC_1)}$.

\smallskip
\begin{lemma}[{\cite[Lemma 6.4]{FRS24}}]
    \label{lem:FRS_6.4}
    Let $g:\R^n \to \R$ satisfy
    $$
    \int_{\R^n}g^2dm = 1 \qquad and \qquad \int_{\R^n}|\nabla g|^2dm \leq 4,
    $$
    where $dm = G(x, -1)dx$ is the Gaussian measure. Then, for some dimensional $R_0 > 0$, we have
    $$
    \int_{B_{R_0}}g^2dm \geq \frac{1}{2}.
    $$
\end{lemma}
\smallskip
\begin{lemma}
    \label{lem:H_leq_w}
    Let $u: B_1 \times (-1, 1) \rightarrow [0, \infty)$ be a bounded solution of \eqref{eq:stefan_problem} and $(0, 0)$ a singular point. Set $w:= (u - p_2)$, let $\gamma \in (2, 3)$ be such that the assumptions of \cref{lem:FRS_5.6} hold for $w\zeta$, let $\lambda^*$ be given by \cref{def:lambdas} and assume $\lambda^* < \gamma$. Then, there exists $r_0 > 0$ such that for all $r \in (0, r_0)$ we have
    $$
    H(r, w\zeta)^{1/2} \leq C\|w_r\|_{L^2(\mathscr{C}_1)},
    $$
    for some constant $C$ depending only on $L$, $f(0)$, $n$ and $\|u(\cdot, 0) \|_{L^\infty(B_1)}$.
\end{lemma}
\begin{proof}
    We will prove the following claim
        \begin{equation}
        \label{eq:claim_lemma_H_leq_w}
            \frac{1}{C}H(r, w\zeta)^{1/2} \leq \|w_r\|_{L^2(\gC_{2R_0})}
        \end{equation}

        for some dimensional $R_0$. Set $\tilde{w}_r:= \frac{(w \zeta)_r}{H(r, w\zeta)^{1/2}}$. By construction we have
        $$
        \int_{ \{t = -1\}} \tilde{w}^2_r G = 1,
        $$
        and
        $$
        \int_{ \{t = -1\}} |\nabla \left(\tilde{w}_r\right)|^2 G = \frac{1}{H(r, w \zeta)} \int_{ \{t = -1\}} |\nabla \left((w\zeta)_r\right)|^2 G = \frac{D(1, (w\zeta)_r)}{2H(r, w\zeta)} = \frac{D(r, w\zeta)}{2H(r, w\zeta)} = \frac{1}{2}\phi(r, w\zeta).
        $$
        
        Since $\lambda^* < \gamma$, by \cref{cor:H_assymptotic_bhvr} for small values of $r$ we have that $\phi(r, w\zeta)$ is close to $\phi^\gamma(r, w\zeta)$. Hence the hypothesis

        $$
        \int_{t = -1}|\nabla\tilde{w}_r|^2G \leq 4
        $$
        is fulfilled.
        By \cref{lem:FRS_6.4} we obtain
        $$
        \int_{ B_{R_0} \times \{t = -1\}} \tilde{w}^2_r G \geq \frac{1}{2}.
        $$
        Repeating the previous argument replacing $r$ by $\theta^{1/2}r$ with $\theta \in (1, 4)$ and rescaling 
        $$
        \int_{ B_{2R_0} \times \{t = -\theta\}} \tilde{w}^2_r G \geq \frac{1}{2}.
        $$
    Integrating the previous inequality in $\{1 \leq \theta \leq 4 \}$ we get
        $$
        \int_{ B_{2R_0} \times \{-4 \leq t \leq -1\}} \tilde{w}^2_r G \geq \frac{3}{2}
        $$
        and the claim \eqref{eq:claim_lemma_H_leq_w} follows since $B_{2R_0} \times \{-4 \leq t \leq -1\} \subset \gC_{2R_0}$.

        Rescaling \eqref{eq:claim_lemma_H_leq_w} we obtain
        $$
        H\left( \frac{r}{2R_0}, w\zeta\right)^{1/2} \leq C\|w_r\|_{L^2(\gC_1)},
        $$

        since we have $\lambda^* < \gamma$, by \cref{cor:H_assymptotic_bhvr}, $H\left( \frac{r}{2R_0}, w\zeta\right)$ and $H\left( r, w\zeta\right)$ are comparable and the result follows.
\end{proof}

\smallskip
\begin{lemma}
    \label{lem:w_and_H_comparable}
    Let $u: B_1 \times (-1, 1) \rightarrow [0, \infty)$ be a bounded solution of \eqref{eq:stefan_problem} and $(0, 0) \in \Sigma_{n-1}$ a singular point in the top stratum. Set $w:= (u - p_2)$, let $\gamma \in (2, 3)$ be such that the assumptions of \cref{lem:FRS_5.6} hold for $w\zeta$, let $\lambda^*$ be given by \cref{def:lambdas} and assume $\lambda^* < \gamma$. Then there exists $r_0 > 0$ such that for all $r \in (0, r_0)$,
    \begin{equation}
        \label{eq:claim_w_leq_H}
        \|w_r\|_{L^2(\mathscr{C}_1)} \leq CH(r, w\zeta)^{1/2}.
    \end{equation}
\end{lemma}
\begin{proof}
     We will argue by contradiction, assume \eqref{eq:claim_w_leq_H} does not hold. 
    Consider a particular subsequence $w_{r_k}$ where each of the $r_k$ is the biggest value such that \eqref{eq:claim_w_leq_H} does not hold taking $C=k$, i.e. we define
    $$
    r_k := \sup \left\{ r > 0: \|w_{r}\|_{L^2(\gC_1)} > k H(r, w\zeta)^{1/2}\right\}.
    $$
    Consider the sequence 
    $$
    v_{r_k} := \frac{w_{r_k}}{\|w_{r_k}\|_{L^2(\gC_1)}},
    $$
    by construction $\|v_{r_k}\|_{L^2(\gC_1)} = 1$. Moreover, by \cref{lem:L2_controls} we have
    \begin{equation}
    \label{eq:problematic_quotient_1}
        \|\nabla v_{r_k}\|_{L^2(\gC_1)} = \frac{\| \nabla w_{r_k}\|_{L^2(\gC_1)}} {\|w_{r_k}\|_{L^2(\gC_1)}}  \leq C\frac{\|w_{r_k}\|_{L^2(\gC_2)} + r_k^3}{\|w_{r_k}\|_{L^2(\gC_1)}}.
    \end{equation}

    The term $\|w_{r_k}\|_{L^2(\gC_2)}/\|w_{r_k}\|_{L^2(\gC_1)}$ is bounded by construction of the subsequence since
    $$
    \frac{1}{2^{(n+2)/2}} \|w_{r_k}\|_{L^2(\gC_2)} = \|w_{2r_k}\|_{L^2(\gC_1)} 
    \leq k H(2r_k, w\zeta)^{1/2}
    \leq kCH(r_k, w\zeta)^{1/2} 
    \leq C\|w_{r_k}\|_{L^2(\gC_1)}.
    $$
    Also $r^3/\|w_r\|_{L^2(\gC_1)} \to 0$ by \cref{lem:H_leq_w} and \cref{cor:H_assymptotic_bhvr}. Hence, the right hand side of \eqref{eq:problematic_quotient_1} is bounded and $v_{r_k} \to q$ strongly in $L^2(\gC_1)$ and weakly in $H^1(\gC_1)$. Since $\|w_{r_k}\|_{L^\infty(\gC_1)}$, $\|\partial_t w_{r_k}\|_{L^\infty(\gC_1)}$ and $[w_{r_k}]_{C^{0, \alpha}_x(\gC_1)}$ are also controlled by $\|w_{r_k}\|_{L^2(\gC_2)} + r_k^3$ (\cref{lem:L2_controls} and \cref{lem:L_infty_bounds}), the sequence $v_{r_k}$ is equicontinuous and we conclude that the convergence to $q$ is uniform in $\gC_1$. 
    
    For any $\tau \in (0, 1)$ consider
    \begin{equation}
        \label{eq:bound_H_tau}
    H(\tau, (v\zeta)_{r_k}) = \frac{1}{\|w_{r_k}\|^2_{L^2(\gC_1)}} H(\tau, (w\zeta)_{r_k}) \leq C\frac{1}{k^2},
    \end{equation}

    where in the last inequality we have used that $H(\tau, (w\zeta)_{r_k})$ and $H(r_k, w\zeta)$ are comparable since $\lambda^* < \gamma$.

    Taking limits when $k \to \infty$ in \eqref{eq:bound_H_tau} we obtain
    $$
    \int_{\{t = -\tau^2\}} q^2 G \leq 0,
    $$

    which is a contradiction since, by construction $\|q\|_{L^2(\gC_1)} = 1$.
\end{proof}

\smallskip
\begin{lemma}
    \label{lem:convergence_to_q}
    Let $u: B_1 \times (-1, 1) \rightarrow [0, \infty)$ be a bounded solution of \eqref{eq:stefan_problem}, $(0, 0) \in \Sigma_{n-1}$. Set $w:= (u - p_2)$, let $\gamma \in (2, 3)$ be such that the assumptions of \cref{lem:FRS_5.6} hold for $w\zeta$, let $\lambda^*$ be given by \cref{def:lambdas} and assume $\lambda^* < \gamma$. Define 

    $$
    \tilde{w}_r := \frac{w_r}{H(r, w\zeta)^{1/2}}.
    $$

    Then, up to subsequences, $\tilde{w}_r \rightarrow q$ locally uniformly in $\R^n\times(-\infty, 0)$, where $q$ is a $\lambda^*$-homogeneous function which is also a solution of the parabolic Signorini problem, i.e. $q$ verifies
    $$
    \begin{cases}
        \bH q \leq 0 \quad \text{and} \quad q \bH q = 0 & \text{in } \R^n \times (-\infty, 0) \\
        \bH q = 0 & \text{in } \R^n \times (-\infty, 0) \setminus \{p_2 = 0 \} \\
        q \geq 0 & \text{on } \{p_2 = 0 \} \\
        \partial_t q \geq 0 & \text{in } \R^n \times (-\infty, 0). 
    \end{cases}
    $$
\end{lemma}
\begin{proof}
    First we will prove convergence (up to subsequence) in $H^1_\text{loc}(\R^n \times (-\infty, 0])$. By the estimate from \cref{lem:L2_controls}, since $H(r, w\zeta)$ is comparable at all scales (\cref{cor:H_assymptotic_bhvr}), it is enough to prove that $\|\tilde{w}_r\|_{L^2(\gC_1)}$ is bounded. For doing so note
    $$
    \|\tilde{w}_r\|_{L^2(\gC_1)} = \frac{\|w_r\|_{L^2(\gC_1)}}{H(r, w\zeta)^{1/2}} \leq C, 
    $$

    and also
    $$
    \|\nabla\tilde{w}_r\|_{L^2(\gC_1)} = \frac{\| \nabla w_r\|_{L^2(\gC_1)}}{H(r, w\zeta)^{1/2}} \leq C \frac{\|w_r\|_{L^2(\gC_2)} + r^3}{H(r, w\zeta)^{1/2}}, 
    $$
    
    since $\lim_{r\downarrow0} \frac{r^3}{H(r, w\zeta)^{1/2}} = 0$ we have $\tilde{w}_r \to q$ for some $q \in H^1_\text{loc}$. Using the $L^\infty$ estimates on $w_r$, $\partial_t w_r$ and $[w_r]_{C^{0, \alpha}_x}$ from \cref{lem:L2_controls} and \cref{lem:L_infty_bounds} we conclude that, up to subsequences, the convergence to $q$ is locally uniformly.

    For checking that $q$ is a solution of the parabolic Signorini problem note 
    $$
    \bH \tilde{w}_r =  \frac{r^2 \left( f(rx) \chi_{\{u > 0\}} - f(0) \right)}{H(r, w\zeta)^{1/2}} \leq \frac{Lr^3}{H(r, w\zeta)^{1/2}} 
    $$
    and recall that 
    $$
    \theta(r) := \frac{Lr^3}{H(r, w\zeta)^{1/2}} \longrightarrow 0 \quad \text{as} \quad r \longrightarrow 0. 
    $$
    Define $\bar{w}_r := \tilde{w}_r - \frac{|x|^2}{2n} \theta(r)$. Hence for $r$ small enough $\bH \bar{w}_r \leq 0 $ and by Riesz representation theorem $\bH \bar{w}_r$ defines a measure which converges to $\bH q$ since $\bar{w}_r$ also converges to $q$ locally uniformly. Moreover it is verified $\bH q \leq 0$. By a similar argument $ \bar{w}_r \bH \bar{w}_r \geq 0$ implies $q\bH q \geq 0$. We conclude that $\bH q$ is supported on $\{p_2 = 0 \}$ since for any set $E \subset \{p_2 > 0\}$ we have $E \subset \{u > 0\}$ which implies $\bH q = 0$.

\end{proof}

The following lemma will be useful to control the $L^2$-norm of $w_r$ when the frequency is saturated.
\smallskip
\begin{lemma}
    \label{lem:bound_for_saturated_frequency}
    Let $u: B_1 \times (-1, 1) \rightarrow [0, \infty)$ be a bounded solution of \eqref{eq:stefan_problem} and $(0, 0) \in \Sigma_{n-1}$ a singular point in the top stratum. Set $w := (u - p_2)$, let $\gamma \in (2, 3)$ be such that the assumptions of \cref{lem:FRS_5.6} hold for $w\zeta$, let $\lambda^*$ be given by \cref{def:lambdas} and assume $\lambda^* = \gamma$. Then, for $r$ small enough it holds
    \begin{equation}
    \label{eq:bound_of_L_2}
        \|w_r\|_{L^2(\gC_1)} \leq C r^{\lambda^* - \delta},    
    \end{equation}
    where $C$ is a constant and $\delta > 0$ is any positive number.

\end{lemma}
\begin{proof}
    We will argue by contradiction, assume \eqref{eq:bound_of_L_2} does not hold. 
    Consider a particular subsequence $w_{r_k}$ where each of the $r_k$ is the biggest value such that \eqref{eq:bound_of_L_2} does not hold (taking $C=k$), i.e. we define
    $$
    r_k := \sup \left\{ r > 0: \|w_{r}\|_{L^2(\gC_1)} \geq k r^{\lambda^* - \delta}\right\}.
    $$
    Consider the sequence 
    $$
    v_{r_k} = \frac{w_{r_k}}{\|w_{r_k}\|_{L^2(\gC_1)}},
    $$
    by construction $\|v_{r_k}\|_{L^2(\gC_1)} = 1$. Moreover, by \cref{lem:L2_controls} we have
    \begin{equation}
    \label{eq:problematic_quotient}
        \|\nabla v_{r_k}\|_{L^2(\gC_1)} = \frac{\| \nabla w_{r_k}\|_{L^2(\gC_1)}} {\|w_{r_k}\|_{L^2(\gC_1)}}  \leq C\frac{\|w_{r_k}\|_{L^2(\gC_2)} + r_k^3}{\|w_{r_k}\|_{L^2(\gC_1)}}.
    \end{equation}

    The term $\|w_{r_k}\|_{L^2(\gC_2)}/\|w_{r_k}\|_{L^2(\gC_1)}$ is bounded by construction of the subsequence since
    $$
    \frac{1}{2^{(n+2)/2}} \|w_{r_k}\|_{L^2(\gC_2)} = \|w_{2r_k}\|_{L^2(\gC_1)} < k\left( 2r_k\right)^{\lambda^* - \delta} \leq 2^{\lambda^* - \delta} \|w_{r_k}\|_{L^2(\gC_1)}.
    $$

    Also we can assume that $r^3/\|w_{r_k}\|_{L^2(\gC_1)}$ is bounded, otherwise the claim of the lemma is trivial. Hence, the right hand side of \eqref{eq:problematic_quotient} is bounded and $v_{r_k} \to q$ strongly in $L^2(\gC_1)$ and weakly in $H^1(\gC_1)$. Since $\|w_{r_k}\|_{L^\infty(\gC_1)}$, $\|\partial_t w_{r_k}\|_{L^\infty(\gC_1)}$ and $[w_{r_k}]_{C^{0, \alpha}_x(\gC_1)}$ are also controlled by $\|w_{r_k}\|_{L^2(\gC_2)} + r_k^3$ (\cref{lem:L2_controls} and \cref{lem:L_infty_bounds}), the sequence $v_{r_k}$ is equicontinuous and we conclude that the convergence to $q$ is uniform in $\gC_1$. 
    
    For any $\tau \in (0, 1)$ consider
    \begin{equation}
        \label{eq:contradiction_with_H}
        H(\tau, v_{r_k}) = \int_{\{t = -\tau^2\}}v_{r_k}^2G = \frac{1}{\|w_{r_k}\|^2_{L^2(\gC_1)}} \int_{\{t = -\tau^2\}}w_{r_k}^2G \leq C_\tau {r_k}^{2\delta},
    \end{equation}
    where in the last inequality we have used $\int_{\{t = -\tau^2\}}w_r^2G = H(\tau r, w) \leq C_\tau r^{2\lambda^*}$. Taking the limit $k \to \infty$ in \eqref{eq:contradiction_with_H} we get

    \begin{equation}
        \label{eq:contradiction_with_q}
        \int_{\{t = -\tau^2\} \cap \gC_1}q^2G = 0,
    \end{equation}

    which is a contradiction since, by construction, ${\|q\|_{L^2(\gC_1)}} = 1$.
\end{proof}

\section{$\phi^\gamma$ is saturated for $\gamma < 3$}
\label{sec:iterative_argument}
In this section we prove that for any value of the parameter $\gamma \in (2, 3)$ we have $\lambda^* = \gamma$. The strategy is to first obtain an improved version of Weiss and Frequency formulas (assuming some extra hypotheses) and then use an iterative argument.

In \cref{sec:weiss_mon_form}, we presented several lemmas that will be useful in the proof of our main theorem. Most of the previous lemmas are formulated assuming $\gamma \in \left(2, \frac{5}{2} \right)$ or $\gamma \in \left(2, \frac{9}{4} \right)$. However there is a crucial intermediate result that needs a refined version of these lemmas with the extra growth hypotheses $\|w_r\|_{L^2(\gC_1)} \leq C r^{2+\alpha}$ and $H(r,w)+D(r,w) \leq C r^{4+2\alpha}$ where $\alpha \in (0, 1)$. This will allow us to obtain similar results for higher values of $\gamma$. 
\smallskip
\begin{lemma}
    \label{lem:cubic_estimates_v2}
    Let $u: B_1 \times (-1,1) \to [0,\infty)$ be a bounded solution of \eqref{eq:stefan_problem}
    with $(0,0) \in \Sigma_{n-1}$. Set $w := (u - p_2)\zeta$. Let $\alpha \in (0,1)$ and assume
    that there exist $C_0 > 0$ and $r_0 > 0$ such that
    \begin{equation*}
        \| w_r \|_{L^2(\gC_1)} \leq C_0\, r^{2+\alpha} \quad \text{and} \quad
        H(r,w)+D(r,w) \leq C_0 r^{4+2\alpha}
        \qquad \text{for all } r \in (0, r_0).
    \end{equation*}
    Then, there exists $\bar{r} > 0$ such that, in the region $\{ u > 0\}$, it holds 
        \begin{equation}
            \label{eq:cubic_v2_positivity}
            \left| \langle w, \bH w \rangle_r^{+} \right| \leq  C r^{3+\alpha}
            \qquad \text{and} \qquad
            \left| \langle Zw, \bH w \rangle_r^{+} \right|
            \leq C r^{3+\alpha} \qquad \text{for all } r \in (0, \bar{r}).
        \end{equation}

    Here $C > 0$ depends only on $L$, $f(0)$, $n$, $\alpha$, $C_0$ and $\|u(\cdot, 0) \|_{L^\infty(B_1)}$.
\end{lemma}

\begin{proof}
    By \cref{lem:exp_error} we may replace $\langle w, \bH w\rangle^+_r$ and $\langle Zw, \bH w\rangle^+_r$ by $\langle \zeta(u-p_2), \zeta\bH(u-p_2)\rangle^+_r$ and $\langle \zeta Z(u-p_2), \zeta\bH(u-p_2)\rangle^+_r$ up to an exponentially small error, which is always absorbed in the polynomial terms; for simplicity we do not write it explicitly.
    
    Note that $\langle g, h \rangle_r = \langle g_r, h_r \rangle_1$ and in $\{u>0\}$, we have $\bH(u-p_2) = f - f(0)$, hence $|(\bH(u-p_2))_r| = |f(rx)-f(0)| \leq Lr|x|$. Using the previous estimate we get 
    \begin{equation*}
        \begin{aligned}
        \left|\langle w_r, (\bH w)_r\rangle^+\right| & \leq \int_{\{u_r(\cdot,-1)>0\}} \left|w_r (\bH w)_r \right|(x, -1)G(x, -1)dx \\
        & \leq Lr \left(\int_{\R^n} |x|^2 G(x, -1)dx\right)^{1/2} \left(\int_{\R^n} w_r^2 G(x, -1)dx\right)^{1/2} \leq Cr H(r, w)^{1/2}
        \leq Cr^{3+\alpha}
        \end{aligned}
    \end{equation*}
    which is the first estimate in \eqref{eq:cubic_v2_positivity}.

    For the second estimate note that $(Zw)_r = Zw_r = x\cdot \nabla w_r + 2t \partial_t w_r$, then $(Zw)_r(x, -1) = x\cdot \nabla w_r - 2 \partial_t w_r$. Using the definition of $\langle \cdot, \cdot \rangle$ we have
    \begin{equation*}
        \begin{aligned}
        \left|\langle (Zw)_r, (\bH w)_r\rangle^+ \right| & \leq \int_{\{u_r(\cdot,-1)>0\}} \left|(x \cdot \nabla w_r - 2 \partial_t w_r) (\bH w)_r \right|(x, -1)G(x, -1)dx \\
        & \leq Lr \int_{\R^n} \left|x(x \cdot \nabla w_r - 2 \partial_t w_r) \right| (x, -1)G(x, -1)dx, 
        \end{aligned}
    \end{equation*}
    and we split it into two integrals. The first one we will split again in two, the idea is for $x$ small use the $L^\infty$ estimate from \cref{lem:L_infty_bounds} and for large $x$ use the exponential decay of $G(x, -1)$. Note that by rescaling, the hypothesis gives
    $\| w_\rho\|_{L^2(\gC_2)} \leq C \rho^{2+\alpha}$ for $\rho < \frac{r_0}{2}$. Hence, replacing $r_0$ by a smaller radius if necessary, we can assume that \cref{lem:L_infty_bounds} holds for all $\rho \in (0, r_0)$.

    For $r \in (0, \frac{r_0}{4})$ and $x \in \R^n$ define $\rho := 2r(1 + |x|)$, note that $r < \rho$ and $r|x| < \rho$. Recall \cref{lem:constant_simplification_bound} tells us $|\partial_t u| \leq C$ in $B_{1/2} \times (-1/2, 0)$. Since $\zeta$ and $p_2$ are time-independent we have 
    $$
    \partial_t w_r(x,t)=r^2 \zeta(rx)(\partial_t u)(rx,r^2 t).
    $$
    At $t=-1$ and $r < \frac{1}{\sqrt{2}}$, either $\zeta(rx) = 0$ or $|rx| \leq 1/2$ and then $(rx, -r^2) \in  B_{1/2} \times (-1/2, 0)$; in both cases we have  
    \begin{equation}
        \label{eq:tmp_improved_cubic_estimates_bad_partial_t_bound}
        |\partial_t w_r(x, -1) | \leq Cr^2 \qquad \text{for all } x \in \R^n.
    \end{equation}
    For an improved estimate, recall that \cref{lem:L_infty_bounds} gives us $\|\partial_t w_r \|_{L^\infty(\gC_{1})} \leq C \|w_r \|_{L^2(\gC_2)} \leq Cr^{2+\alpha}$ for $r\in (0, r_0)$, which rescaled implies 
    \begin{equation}
        \label{eq:tmp_improved_cubic_estimates_good_partial_t_bound}
        \sup_{\gC_r}|\partial_t w| \leq Cr^\alpha \qquad \text{for all } r\in (0, r_0).
    \end{equation}
    Now we split our integral in the regions $\left\{ |x| < \frac{r_0}{2r} - 1 \right\}$ (which is $\rho < r_0$) and $\left\{|x| \geq \frac{r_0}{2r} - 1 \right\}$ (which is $\rho \geq r_0$)
    \begin{equation*}
        \int_{\R^n} |x|\cdot|\partial_t w_r|G(x, -1)dx =  \int_{\left\{ |x| < \frac{r_0}{2r} - 1 \right\}} |x|\cdot|\partial_t w_r|G(x, -1)dx + \int_{\left\{|x| \geq \frac{r_0}{2r} - 1 \right\}} |x|\cdot|\partial_t w_r|G(x, -1)dx
    \end{equation*}
    and we bound each term separately. In the region ${\left\{ |x| < \frac{r_0}{2r} - 1 \right\}}$ we have $(rx, -r^2) \in \gC_{\rho}$, hence \eqref{eq:tmp_improved_cubic_estimates_good_partial_t_bound} holds and we can bound
    \begin{equation*}
        \begin{aligned}
        \int_{\left\{ |x| < \frac{r_0}{2r} - 1 \right\}} |x|\cdot|\partial_t w_r|G(x, -1)dx & = r^2 \int_{\left\{ |x| < \frac{r_0}{2r} - 1 \right\}} |x|\cdot|\partial_t w|(rx, -r^2)G(x, -1)dx \\
        & \leq Cr^2 \int_{\left\{ |x| < \frac{r_0}{2r} - 1 \right\}} |x|\rho^\alpha G(x, -1)dx \\
        & \leq Cr^{2+\alpha} \int_{\left\{ |x| < \frac{r_0}{2r} - 1 \right\}} |x|(1+ |x|)^\alpha G(x, -1)dx
        \leq C r^{2+ \alpha},
        \end{aligned}
    \end{equation*}
    where in the last inequality we have also used $\int_{\R^n}|x|(1+ |x|)^\alpha G < C$.
    In the region ${\left\{ |x| \geq \frac{r_0}{2r} - 1 \right\}}$, since $r < \frac{r_0}{4}$, we have that $|x| \geq \frac{r_0}{2r}-1 \geq \frac{r_0}{4r}$ which implies 
    $$
        G(x, -1) = \frac{1}{(4\pi)^{n/2}}e^{-|x|^2/4} =  \frac{1}{(4\pi)^{n/2}}e^{-|x|^2/8}\cdot e^{-|x|^2/8} \leq \frac{ e^{-r_0^2/(128 r^2)}}{(4\pi)^{n/2}}e^{-|x|^2/8}.
    $$
    Combining this with the bound from \eqref{eq:tmp_improved_cubic_estimates_bad_partial_t_bound} we get
    \begin{equation*}
        \begin{aligned}
        \int_{\left\{ |x| \geq \frac{r_0}{2r} - 1 \right\}}|x| \cdot |\partial_t w_r|G(x, -1)dx &\leq Cr^2 \int_{\left\{ |x| 
        \geq \frac{r_0}{2r} - 1 \right\}} |x|G(x, -1)dx \\ &\leq Cr^2 e^{-r_0^2/(128 r^2)}\int_{\left\{ |x| \geq \frac{r_0}{2r} - 1 \right\}} |x|e^{-|x|^2/8}dx \leq C r^3.
        \end{aligned}
    \end{equation*}
    Combining everything together we get the bound 
    \begin{equation*}
        \int_{\R^n} |x|\cdot|\partial_t w_r|G(x, -1)dx \leq Cr^{2+\alpha}.
    \end{equation*}
    For the second integral we use Cauchy-Schwarz to get
    \begin{equation*}
        \begin{aligned}
            \int_{\R^n} \left|x \cdot (x \cdot \nabla w_r) \right| (x, -1)G(x, -1)dx & \leq \left(\int_{\R^n} |x|^4 G(x, -1)dx\right)^{1/2} \left(\int_{\R^n} (\nabla w_r)^2 G(x, -1)dx\right)^{1/2}\\ &\leq CD(r, w)^{1/2} .
        \end{aligned}
    \end{equation*}

    From the previous two inequalities we get 
    $$
    \left| \langle Zw, \bH w\rangle_r^+  \right| \leq Cr \left(r^{2+\alpha} + D(r, w)^{1/2} \right) \leq C r^{3+\alpha},
    $$
    which is the second estimate from \eqref{eq:cubic_v2_positivity}.
\end{proof}

\begin{remark}
    The bounds in the contact set \cref{lem:contact_set_FRS_5.1} are still valid and useful without any extra hypothesis. 
\end{remark}
Now we repeat all the previous arguments using the new bound from \cref{lem:cubic_estimates_v2}.
\smallskip
\begin{lemma}[Weiss'-type almost monotonicity formula]
\label{lem:weiss_type_formula_v2}
 Let $u: B_1 \times (-1,1) \to [0,\infty)$ be a bounded solution of \eqref{eq:stefan_problem}
    with $(0,0) \in \Sigma_{n-1}$. Set $w := (u - p_2)\zeta$. Let $\alpha \in (0,1)$ and assume that there exist $C_0 > 0$ and $r_0 > 0$ such that
    \begin{equation*}
        \| w_r \|_{L^2(\gC_1)} \leq C_0\, r^{2+\alpha} \quad \text{and} \quad
        H(r,w)+D(r,w) \leq C_0 r^{4+2\alpha}
        \qquad \text{for all } r \in (0, r_0).
    \end{equation*}
$$
W(r, w) := \frac{1}{r^4}(D(r, w) - 2 H(r, w)).
$$
Then, there exists $\bar{r} > 0$ such that:
\begin{enumerate}
    \item For all $r \in (0, \bar{r})$
    $$
    \diff{ \ }{r}W(r, w) \geq -Cr^{\alpha},
    $$
    where $C$ depends only on $L$, $f(0)$, $n$, $\alpha$, $C_0$ and $\|u(\cdot, 0) \|_{L^\infty(B_1)}$.
    \item $W(0^+, w) = 0$.
    
    \item For all $r \in (0, \bar{r})$
    $$
    -C r^{5 +\alpha} \leq D(r, w) - 2H(r, w) \leq Cr^{4}
    $$
    where $C$ depends only on $L$, $f(0)$, $n$, $\alpha$, $C_0$ and $\|u(\cdot, 0) \|_{L^\infty(B_1)}$.
\end{enumerate}
\end{lemma}

\begin{proof}\textit{(a)} Following the computations from \cref{lem:weiss_type_formula} and using in this case \eqref{eq:cubic_v2_positivity} we get
    \begin{equation}
        \label{eq:weiss_formula_v2}
    \begin{aligned}
        \diff{ \ }{r}W(r, w) &= \frac{2}{r^5}(\langle Zw - 2w, Zw -2w \rangle_r + 2r^2 \langle 2w -Zw, \bH w \rangle_r )
        \geq -Cr^{\alpha}. 
    \end{aligned}
    \end{equation}
    In the contact set $\{u = 0 \}$, \cref{lem:contact_set_FRS_5.1} holds and (exactly as in \cref{lem:weiss_type_formula}) we have
    $$
    \diff{ \ }{r}W(r, w) \geq  -C e^{-\frac{1}{r^2}}
    $$
    In the region $\{u > 0\}$, using the improved estimates from \cref{lem:cubic_estimates_v2}, we obtain
    $$
    \diff{ \ }{r}W(r, w) \geq  -C r^\alpha.
    $$
\textit{(b)} Is similar to \cref{lem:weiss_type_formula} (b).
    
\textit{(c)} The upper bound is exactly as in \cref{lem:weiss_type_formula} c). The lower bound is a direct consequence of (a) and (b), integrating \eqref{eq:weiss_formula_v2} from $0$ to $r$ we obtain
    $$
    W(r, w) = \frac{1}{r^4} \left(D(r, w) -2H(r, w) \right) \geq -Cr^{ 1+\alpha},
    $$
    and the claim follows.
\end{proof}

\smallskip
\begin{lemma}
    \label{lem:freq_direct_bound_v2}
     Let $u: B_1 \times (-1,1) \to [0,\infty)$ be a bounded solution of \eqref{eq:stefan_problem}
    with $(0,0) \in \Sigma_{n-1}$. Set $w := (u - p_2)\zeta$. Let $\alpha \in (0,1)$ and assume
    that there exist $C_0 > 0$ and $r_0 > 0$ such that
    \begin{equation*}
        \| w_r \|_{L^2(\gC_1)} \leq C_0 r^{2+\alpha} \quad \text{and} \quad
        H(r,w)+D(r,w) \leq C_0 r^{4+2\alpha}
        \qquad \text{for all } r \in (0, r_0).
    \end{equation*}
    Then, for any given $\gamma \in \left(2 + \alpha, \frac{5 + \alpha}{2} \right)$, there exists $\bar{r}$ such that, for all $r \in (0, \bar{r})$, it holds 
    \begin{equation}
       0 \leq \phi^{\gamma}(r, w) - 2 \leq Cr^{4 +2\alpha - 2\gamma},
    \end{equation}
    where $C$ depends only on $L$, $f(0)$, $n$, $\alpha$, $C_0$ and $\|u(\cdot, 0) \|_{L^\infty(B_1)}$.
\end{lemma}
\begin{proof}
    For the upper bound we use the hypothesis $H(r,w)+D(r,w) \leq C r^{4+2\alpha}$, since $\gamma > 2 + \alpha$, we get
    \begin{equation*}
        \begin{aligned}
            \phi^\gamma(r, w) - 2 & = \frac{D(r, w) - 2H(r, w) + (\gamma - 2)r^{2\gamma}}{H(r, w) + r^{2\gamma}} \\
            & \leq \frac{Cr^{4+2\alpha} + (\gamma - 2)r^{2\gamma}}{H(r, w) + r^{2\gamma}} \leq C r^{4+2\alpha-2\gamma}.
        \end{aligned}
    \end{equation*}For the lower bound following the same steps as in \cref{lem:freq_direct_bound} but using the improved version of Weiss \cref{lem:weiss_type_formula_v2} we get
    \begin{equation*}
        \begin{aligned}
            \phi^\gamma(r, w) - 2 & = \frac{D(r, w) - 2H(r, w) + (\gamma - 2)r^{2\gamma}}{H(r, w) + r^{2\gamma}} \\
            & \geq \frac{-Cr^{5+\alpha} + (\gamma - 2)r^{2\gamma}}{H(r, w) + r^{2\gamma}} \\
            & = \frac{ r^{2\gamma}\left(-Cr^{5 + \alpha-2\gamma} + \gamma - 2 \right)}{H(r, w) + r^{2\gamma}} \geq 0,
        \end{aligned}
    \end{equation*}
    as wanted.
\end{proof}

\smallskip
\begin{lemma}
    \label{lem:frequency_formula_monotonicity_v2}
    Let $u: B_1 \times (-1,1) \to [0,\infty)$ be a bounded solution of \eqref{eq:stefan_problem}
    with $(0,0) \in \Sigma_{n-1}$. Set $w := (u - p_2)\zeta$. Let $\alpha \in (0,1)$ and assume
    that there exist $C_0 > 0$ and $r_0 > 0$ such that
    \begin{equation*}
        \| w_r \|_{L^2(\gC_1)} \leq C_0 r^{2+\alpha} \quad \text{and} \quad
        H(r,w)+D(r,w) \leq C_0 r^{4+2\alpha}
        \qquad \text{for all } r \in (0, r_0).
    \end{equation*}
    Then, for any given $\gamma \in \left(2 + \alpha, \frac{9 + 3\alpha}{4} \right)$, there exists $\bar{r} > 0$ such that for all $r \in (0, \bar{r})$, it holds
    \begin{equation}
        \label{eq:frequency_formula_monotonicity_v2}
        \diff{ }{r} \phi^{\gamma} (r, w)\geq \frac{2}{r} \left(\frac{2r^2 \langle w, \bH w\rangle_r}{H(r, w) + r^{2\gamma}}\right)^2 -Cr^{8 + 3\alpha - 4\gamma}
    \end{equation}

    where $C$ is a constant depending only on $L$, $f(0)$, $n$, $\alpha$, $C_0$ and $\|u(\cdot, 0) \|_{L^\infty(B_1)}$.
\end{lemma}
\begin{proof}
    Repeating the computations from \cref{lem:frequency_formula_monotonicity} we get
    \begin{align*}
        \frac{E^\gamma(r, w)}{r\left(H(r, w) + r^{2\gamma}\right)^2} &=\frac{2r \langle w, \bH w \rangle_r \left(\phi^\gamma(r, w) -2\right) + 2r\langle2w - Zw, \bH w \rangle_r}{H(r, w) + r^{2\gamma}}.
    \end{align*}
     In the contact set $\{u = 0\}$, by \cref{lem:contact_set_FRS_5.1} we have the bound $2r \langle w, \bH w \rangle_r \left(\phi^\gamma(r, w) -2\right) + 2r\langle2w - Zw, \bH w \rangle_r \geq -Cre^{-\frac{1}{r^2}} \geq -Cr^4$ (exactly as in \cref{lem:frequency_formula_monotonicity}).

     In the region $\{u > 0\}$, combining the improved estimates from \cref{lem:cubic_estimates_v2} with \cref{lem:freq_direct_bound_v2} we get     
     $2r \langle w, \bH w \rangle_r \left(\phi^\gamma(r, w) -2\right) \geq -Cr^{8+3\alpha-2\gamma}$ and $2r\langle2w - Zw, \bH w \rangle_r \geq -Cr^{4+\alpha}$, hence
     $$
        \frac{E^\gamma(r, w)}{r\left(H(r, w) + r^{2\gamma}\right)^2}  \geq -C(r^{8+3\alpha-4\gamma} + r^{4+\alpha-2\gamma}) \geq -Cr^{8+3\alpha-4\gamma},
     $$
     since $8+3\alpha-4\gamma < 4+\alpha-2\gamma$.
\end{proof}

\smallskip
\begin{proposition}
\label{prop:lamb=gam}
Let $u: B_1 \times (-1, 1) \rightarrow [0, \infty)$ be a bounded solution of \eqref{eq:stefan_problem}, $(0, 0) \in \Sigma_{n-1}$. Let $\lambda^*$ be given by \cref{def:lambdas}. For any $\gamma \in (2, 3)$ it verifies $\lambda^* = \gamma$.
\end{proposition}

\begin{proof}
    We will divide the proof in two steps. In the first one we will prove the claim for the case $\gamma \in \left(2, \frac{9}{4} \right)$ and in the second we will use an iterative argument. 

        \par 
        \noindent\textbullet\ \textit{Step 1.} Take $\gamma_1 \in \left(2, 3 - \frac{3}{4} \right) = \left(2, \frac{9}{4} \right)$ and assume $\lambda_1^* < \gamma_1$ by \cref{lem:convergence_to_q} $\tilde{w}_r$ converges to a $\lambda_1^*$-homogeneous solution of the Signorini problem. By \cite[Theorem 3.1]{Col25} there are no nonzero solutions of the Signorini problem with homogeneity strictly between 2 and 3, then $\lambda^*_1 = 2$. However, we next show that this case is also not possible.

        Given $p \in \cP$, set $w := u - p_2$ and $\tilde{w}_r = \frac{w_r}{H(r, w \zeta)^{1/2}}$. Since $\gamma_1 \in \left(2, \frac{9}{4} \right)$ the almost-monotonicity of the frequency formula \cref{lem:frequency_formula_monotonicity} holds and then, by \cref{lem:monotonicity_H}, applied to $(u-p)\zeta$, we have $H(r, (u-p)\zeta)/r^4 +Cr$ is non-decreasing. Hence,
        $$
        \frac{1}{r^4} \int_{\{t = -1\}} \left((u-p)_r\zeta_r\right)^2G + Cr\geq \lim_{r \downarrow 0} \left( \frac{H(r, (u-p)\zeta)}{r^4} +Cr \right) = \int_{\{t = -1\}} (p_2-p)^2G,
        $$
        and
        $$
        Cr + \int_{\{t = -1\}} \left( \left( \frac{w_r}{r^2} + p_2 -p\right)\zeta_r \right)^2G = Cr+ \int_{\{t = -1\}} \left( \left( \frac{u_r - p_r}{r^2}\right)\zeta_r \right)^2G \geq  \int_{\{t = -1\}} (p_2-p)^2G.
        $$
        Define $\tilde{\eps}_r := \frac{H(r, w \zeta)^{1/2}}{r^2}$, note that $\tilde{\eps}_r \to 0$ as $r \to 0$ this is only verified for $w = u - p_2$, then
        $$
        Cr + \int_{\{t = -1\}} (\tilde{\eps}_r \tilde{w}_r + p_2 - p)^2\zeta_r^2G \geq \int_{\{t = -1\}} (p_2 - p)^2G, 
        $$
    
        expanding terms we obtain
        \begin{equation}
            \label{eq:eps_formula_orth}
            \tilde{\eps}_r^2\int_{\{t = -1\}} (\tilde{w}_r\zeta_r)^2G + 2 \tilde{\eps}_r \int_{\{t = -1\}} \tilde{w}_r(p_2 - p)\zeta_r^2G + Cr\geq 0.
        \end{equation}
    
        Dividing \eqref{eq:eps_formula_orth} by $\tilde{\eps}_r$ and taking $r \downarrow0$, since $r/\tilde{\eps}_r \to 0$ for $\lambda^*_1 < \gamma_1$ (recall \cref{cor:H_assymptotic_bhvr}) we get
        \begin{equation}
            \label{eq:orthogonality}
            \int_{\{t = -1\}} q(p_2 - p)G \geq 0,
        \end{equation}
        since there are no nonzero 2-homogeneous solutions to the parabolic Signorini problem verifying \eqref{eq:orthogonality} (see \cite[Proposition 6.7 b)]{FRS24})we conclude $\lambda^*_1 > 2$. 
        Hence, $\lambda^*_1 = \gamma_1$.
    
    \par 
    \noindent\textbullet\ \textit{Step 2.} 
    From \textit{Step 1} we know that, for any $\gamma_1 \in \left(2, \frac{9}{4} \right)$ it holds $\lambda^*_1 = \gamma_1$. Hence, by \cref{lem:frequency_formula_monotonicity} and the estimates from \cref{lem:cubic_estimates} the inequalities
        $$
        \diff{}{r} \phi^{\gamma_1}(r, w) \geq \frac{2}{r} \frac{(2r^2 \langle w, \bH w\rangle_r)^2}{(H(r, w) + r^{2\gamma_1})^2} -Cr^{\eps_1 - 1} \qquad \text{and} \qquad \frac{2r^2\langle w, \bH w \rangle_r}{H(r, w) + r^{2\gamma_1}} \geq -Cr^{\eps_1}
        $$
    
        are fulfilled with $\eps_1 = 9 -4\gamma_1 > 0$. Therefore, we can use \cref{cor:H_and_D_bound} to obtain
        \begin{equation}
            \label{eq:tmp_improved_regularity_iteration}
            H(r, w) + D(r, w )\leq Cr^{2\lambda^*_1}.
        \end{equation}
        and, by \cref{lem:bound_for_saturated_frequency} we also have
        $$
        \|w_r\|_{L^2(\gC_1)} \leq Cr^{\lambda^*_1 - \delta},
        $$
        Which are the extra hypotheses of \cref{lem:freq_direct_bound_v2} and \cref{lem:frequency_formula_monotonicity_v2} with $\alpha := \lambda^*_1 - \delta -2$. Note that, at this first step, $\alpha \in \left(0, \frac{1}{4} \right)$, and also $\gamma_1 = 2 + \alpha + \delta > 2 +\alpha $. Then, for any $\gamma_2 \in  \left(\gamma_1, \frac{9+3\alpha}{4} \right)$, we have $\gamma_2 > 2 +\alpha$ and the interval $\left(\gamma_1, \frac{9+3\alpha}{4} \right)$ is non empty provided that $3 - 3\delta > \gamma_1$ (which can be always guaranteed taking $\delta$ small enough).

        Therefore we fulfill all the extra hypotheses from Lemmas \ref{lem:cubic_estimates_v2}, \ref{lem:weiss_type_formula_v2}, \ref{lem:freq_direct_bound_v2} and \ref{lem:frequency_formula_monotonicity_v2}. Then, by \cref{lem:frequency_formula_monotonicity_v2}, we have
        $$
        \diff{}{r}\phi^{\gamma_2}(r, w) \geq  \frac{2}{r} \left(\frac{2r^2 \langle w, \bH w\rangle_r}{H(r, w) + r^{2\gamma_2}}\right)^2 -Cr^{8 + 3\alpha - 4\gamma_2}
        $$
        and, by \cref{lem:cubic_estimates_v2}, we have
        $$
            \frac{2r^2 \langle w, \bH w\rangle_r}{H(r, w) + r^{2\gamma_2}} \geq -Cr^{5 + \alpha - 2\gamma_2} \geq -Cr^{9 + 3\alpha - 4\gamma_2},
        $$
        where we have used $5 + \alpha - 2\gamma_2 > 9 + 3\alpha - 4\gamma_2$ since $\gamma_2 > 2 + \alpha$. Which means that both hypotheses from \cref{lem:FRS_5.6} hold with $\eps_2 := 9 + 3\alpha - 4\gamma_2$. In particular, by \cref{lem:freq_direct_bound_v2} and \cref{lem:frequency_formula_monotonicity_v2}, $\lambda_2^* := \phi^{\gamma_2}(0^+, w)$ is well defined. By the same reasons, all the corollaries of \cref{lem:FRS_5.6}, as well as \cref{lem:convergence_to_q}, hold for $\gamma_2 \in \left(\gamma_1, \frac{9 + 3\alpha}{4} \right)$. Arguing as in \textit{Step 1}, since our reasoning is valid for any $\gamma_1 \in \left(2, \frac{9}{4} \right)$ and $\delta$ is arbitrarily small we conclude that for any $\gamma_2 \in \left(2, 3 - \left(\frac{3}{4}\right)^2\right) =\left(2, \frac{39}{16}\right)$ the frequency formula $\phi^{\gamma_2}(r, w)$ is saturated, i.e. $\lambda^*_2 = \gamma_2$. The new upper bound
        appears since we can choose $\gamma_1$ and $\delta$ such that $\alpha$ is arbitrarily close to $1/4$ and then $\frac{9+(3/4)}{4} =  3 - \left( \frac{3}{4}\right)^2 = \frac{39}{16}$.

        The same argument can now be repeated starting from $\gamma_2$ instead of $\gamma_1$. Indeed, if $\lambda^*_k = \gamma_k$ for every $\gamma_k \in \left(2, a_k\right)$, taking $\alpha := \gamma_k - 2 - \delta$ (so that $\alpha$ can be taken arbitrarily close to $a_k - 2$) the above argument gives $\lambda^*_{k+1} = \gamma_{k+1}$ for every 
        $$
        \gamma_{k+1} \in \left(2, a_{k+1}\right), \qquad a_{k+1} := \frac{9 + 3(a_k - 2)}{4} = \frac{3 + 3a_k}{4}.$$

Since $a_1 = \frac94 = 3 - \frac34$ and $3 - a_{k+1} = \frac{3}{4}\left(3 - a_k\right)$, we obtain $a_k = 3 - \left(\frac34\right)^k$. Then, for any $\gamma \in (2,3)$, choosing $k_0$ such that $\gamma < 3 - \left(\frac34\right)^{k_0}$, we get $\lambda^* = \gamma$ as wanted.
\end{proof}
\section{Proof of main result}
\label{sec:proof_of_main_thm}
In this section, we will prove \cref{thm:main_theorem}. To do so we follow a similar strategy that in Section 8 of \cite{FRS24}. We will show that for every $(x_0, t_0) \in \Sigma$ and every $\eps > 0$ small enough, there exists $r_0 = r_0(x_0, t_0, \eps) > 0$ such that
\begin{equation}
    \label{eq:quadratic_cleaning}
    \Sigma \cap \left\{ |x - x_0| \leq r, \ t-t_0 \geq  r^{2 - \eps}\right\} = \emptyset \qquad \text{for all } r \in (0, r_0).
\end{equation}

The previous result combined with a technical GMT lemma will allow us to prove the main theorem. To prove \eqref{eq:quadratic_cleaning} we first need the following lemma.

Note that, since $f$ does not depend on $t$, we have $\bH(\partial_t u) = 0$ in the open set $\{u > 0\}$, i.e.\ $\partial_t u$ is caloric there. Globally, $\bH(\partial_t u) = f\,\partial_t \chi_{\{u>0\}} \geq 0$ is a nonnegative measure, since $\{u > 0\}$ is nondecreasing in time.

\smallskip
\begin{lemma}
    \label{lem:quadratic_cleaning}
    Let $u: B_1 \times (-1, 1) \rightarrow [0, \infty)$ be a bounded solution of \eqref{eq:stefan_problem}, $(0, 0)$ a singular point. Assume $H(r, u - p_2)^{1/2} \leq \omega(r)$, where $\omega(r) = o(r^2)$ as $r\downarrow0$, $\textbf{e}_n$ is an eigenvector of $D^2p_2$ with maximal eigenvalue and that there exists $c > 0$ such that
    $$
    \fint_{B_r \cap \{|x_n| \geq \frac{r}{10}\}} \partial_tu(\cdot, -r^2) \geq c r^\beta \qquad \text{for all } r \in (0, r_0), \text{ for some } \beta \in (0, 1].
    $$
    Then, there exists a constant $C$ such that for all $r$ small enough we have
    $$
    \{u = 0\} \cap \left(B_{r/2} \times \left[C(\omega(r)r^{-\beta} + r^2), 1\right) \right) = \emptyset.
    $$
\end{lemma}

\begin{proof}
    The proof is a minor modification of that of \cite[Lemma 8.1]{FRS24}, with constants depending also on $L$ and $f(0)$, we only describe the changes.
    \begin{enumerate}
        \item[(i)]  The estimate $\partial_{tt}u \geq -C$ of \cite[Proposition 3.4]{FRS24} holds with the same proof: since $\partial_t f \equiv 0$, for $h > 0$ we have $\bH(\delta_{(t,h)}u) = f\,\delta_{(t,h)}\chi_{\{u>0\}} \geq 0$, which vanishes in $\{u>0\}$ and is bounded in $L^1$ uniformly in $h$, while the continuity of $\partial_t u$ (\cite[Proposition 3.3]{FRS24} for $f \equiv 1$) is given by \cite[Proposition 6.1]{AK23}.
        \item[(ii)]  Since $\bH(u - p_2) = f\chi_{\{u>0\}} - f(0) \leq L|x|$, the function $u - p_2$ is not supercaloric, but $u - p_2 - \frac{Lr}{n}|x|^2$ is supercaloric in $B_{2r}\times(-1,1)$ and nondecreasing in time. Arguing as in \cite{FRS24} with this function, we obtain \cite[(8.3)]{FRS24} with $\tilde\omega(r) := \omega(r) + Lr^3 = o(r^2)$ in place of $\omega(r)$, and we may take $C_1 \geq 2$ there.
        \item[(iii)] We replace the barrier of \cite{FRS24} by
        $$
        \psi(x,t) := p_2(x) + C_1\tilde\omega(r)\left(-1 + h^\delta\left(\tfrac{x}{r}, \tfrac{t-T}{r^2}\right)\right) + \tfrac{Lr}{2n}\left(|x|^2 - r^2\right), \qquad T := 2C_1c_2^{-1}\tilde\omega(r)r^{-\beta},
        $$
        with $h^\delta$ and $c_2$ as in \cite{FRS24}. Then $\bH\psi = f(0) + Lr \geq \bH u$ in $B_r \times (T, 1)$, and the last term vanishes on $\partial B_r$ and is nonpositive in $B_r$, so the comparison on the parabolic boundary is as in \cite{FRS24}. Moreover, since $C_1 \geq 2$ and $\tilde\omega(r) \geq Lr^3$, the last term is bounded below by $-\frac{L}{2n}r^3 \geq -\frac{C_1}{4}\tilde\omega(r)$, so the final estimate of \cite{FRS24} becomes
        $$
        u \geq \psi \geq p_2 + \tfrac{C_1}{4}\tilde\omega(r) > 0 \qquad \text{in } B_{r/2}\times\left[T + r^2, 1\right).
        $$
    \end{enumerate}
    Finally, since $\beta \leq 1$, we have $T + r^2 \leq C\left(\omega(r)r^{-\beta} + r^2\right)$, and the lemma follows.
\end{proof}

We will use the previous lemma in two different cases:

\begin{itemize}
    \item If $(x_0, t_0) \in \Sigma_m$ for $m \leq n-2$, then the expansion in \eqref{eq:intro_basic_expansion} is tight. However, we can obtain the quadratic cleaning directly using $\omega(r) = o(r^2)$ and $\beta = \eps$.

    \item If $(x_0, t_0) \in \Sigma_{n-1}$ we can improve the expansion \eqref{eq:intro_basic_expansion} and take $\omega(r) = r^{3 - \eps}$ and $\beta = 1$.
\end{itemize}

For the case $(x_0, t_0) \in \Sigma_m$ we need the following auxiliary lemma:
\smallskip
\begin{lemma}[{\cite[Lemma 5.8]{FRS24}(see also \cite{Tor24})}]
    \label{lem:FRS_5.8_claim}
    Let $V$ be a linear subspace with $\text{dim}(V) \leq n-2 $. For any $\delta > 0$, define
    $$
    \cC_{\delta} := \left\{ (x, t) \in \R^n\times(-\infty, 0): \text{dist}(x, V) \geq \delta\left(|x| + |t|^{1/2} \right) \right\}.
    $$
    
    For any $\epsilon > 0$ there exists $\delta$ such that the following holds: there exists $N \in (0, \epsilon)$ and a positive $N$-homogeneous function $\Phi:\R^n \times (-\infty, 0) \to (0, +\infty)$ such that
    $$
    \begin{cases}
        \bH \Phi = 0 & \text{in } \cC_\delta \\
        \Phi = 0 & \text{on } \partial\cC_\delta.
    \end{cases}
    $$
\end{lemma}

\smallskip
\begin{lemma}
    \label{lem:cuad_cleaning_Sm}
    Let $u: B_1 \times (-1, 1) \rightarrow [0, \infty)$ be a bounded solution of \eqref{eq:stefan_problem} and $(0, 0) \in \Sigma_m$ with $m \leq n - 2$. Then, for any $\eps > 0$ there exists $r_\eps > 0$ such that
    $$
    \{u=0 \} \cap \left(B_r \times [r^{2-\eps}, 1) \right) = \emptyset
    $$
    for all $r \in (0, r_\eps)$.
\end{lemma}
\begin{proof}
    We will show for $r$ small enough
    \begin{equation}
        \label{eq:botom_stratums_condition}
        \fint_{\partial B_r} \partial_tu(x, t)dx \geq c_\eps r^\eps \qquad \text{for all } t \in (-r^2, 0).
    \end{equation}

    Define $\cC_{\delta}$ as in \cref{lem:FRS_5.8_claim} with $V := \{p_2 = 0\}$, which satisfies $\dim V = m \leq n-2$, and note that (since $r^{-2}u_r -p_2 \to 0$ locally uniformly) for $\delta$ small enough there exists $r_\delta$ such that
    $$
    \overline{\cC}_\delta \cap \partial_\text{par} \gC_{r_\delta} \subset \subset \{u > 0 \}.
    $$
    Since $\partial_tu$ is caloric and positive inside $\{ u > 0\}$ we can use the function $c\Phi$ for some $c> 0$ given by \cref{lem:FRS_5.8_claim} applied with $\epsilon = \eps$ as a barrier.
    $$
    \begin{cases}
        \bH \Phi = \bH \partial_t u = 0 & \text{in } \cC_{\delta} \cap \gC_{r_\delta} \\
        \partial_tu > c\Phi& \text{on } \partial_\text{par}(\cC_{\delta} \cap \gC_{r_\delta}),
    \end{cases}
    $$
    by comparison principle $\partial_tu \geq c\Phi$ in $\cC_{\delta} \cap \gC_{r_\delta}$. Hence
    \begin{equation*}
        \begin{aligned}
          \fint_{\partial B_r} \partial_t u(x, t) dx &\geq
    \fint_{\partial B_r \cap \cC_\delta} \partial_t u(x, t) dx \\ &\geq
    c\fint_{\partial B_r \cap \cC_\delta} \Phi(x, t) dx =
    c\fint_{\partial B_1 \cap \cC_\delta} \Phi(rx, r^2 t) dx =
    cr^N\fint_{\partial B_1 \cap \cC_\delta} \Phi(x, t) dx \geq c_\eps r^\eps,
        \end{aligned}
    \end{equation*}
    which proves the claim \eqref{eq:botom_stratums_condition}. Applying \cref{lem:quadratic_cleaning} with $\omega(r) = o(r^2)$ and $\beta = \eps$ the result follows. 
\end{proof}

The result for the case $(x_0, t_0) \in \Sigma_{n-1}$ is the following.
\smallskip
\begin{lemma}
    \label{lem:cuad_cleaning_S3_n-1}
    Let $u: B_1 \times (-1, 1) \rightarrow [0, \infty)$ be a bounded solution of \eqref{eq:stefan_problem} and $(0, 0) \in \Sigma_{n-1}$. Assume without loss of generality $p_2 = \frac{f(0)}{2}x_n^2$. Then, for any $\eps > 0$ there exists $r_\eps > 0$ such that
    $$
    \{u=0 \} \cap \left(B_r \times [r^{2-\eps}, 1) \right) = \emptyset
    $$
    for all $r \in (0, r_\eps)$.
\end{lemma}
In order to prove it we will need two auxiliary results.
\smallskip
\begin{lemma} 
\label{lem:equiv_to_FRS6.6} Let $u: B_1 \times (-1, 1) \rightarrow [0, \infty)$ be a bounded solution of \eqref{eq:stefan_problem}, $(0, 0) \in \Sigma_{n-1}$. Assume without loss of generality $p_2 = \frac{f(0)}{2}x_n^2$. Set $w := (u - p_2)$. Then for any $\eps > 0$ there exists $r_0 > 0$ such that 

    $$
    \{u(\cdot, t) = 0\} \cap B_r \subset \left\{ |x_n| \leq Cr^{2 - \eps}\right\}
    $$
    for all $r \in (0, r_0)$ and $t \geq -r^2$.
\end{lemma}

\begin{proof}
    Take $\gamma$, arbitrarily close to 3, fixed, by \cref{lem:L_infty_bounds} for all $\alpha \in (0, 1)$ we have
    $$
    [w_r]_{C_x^{0, \alpha}(\gC_1)} = r^\alpha [w]_{C_x^{0, \alpha}(\gC_r)} \leq C\left(\|w_r\|_{L^2(\gC_2)} + r^3 \right).
    $$

    Let $x \in \{u(\cdot, -r^2) = 0\} \cap B_r$ and let $\bar x := x - x_n \textbf{e}_n$ be its projection onto $\{p_2 = 0\}$, so that $\bar x \in B_r$. Since $u \geq 0$ and $u(x, -r^2) = 0$, we have $w(x, -r^2) = -\frac{f(0)}{2}x_n^2$ and $w(\bar x, -r^2) = u(\bar x, -r^2) \geq 0$, whence
    $$
    \frac{f(0)}{2}|x_n|^{2} \leq w(\bar x, -r^2) - w(x, -r^2) \leq [w]_{C_x^{0, \alpha}(B_r \times \{-r^2\})}\,|x_n|^{\alpha}.
    $$
    Combining this with the previous estimate and \cref{lem:bound_for_saturated_frequency}, for any $\delta > 0$ we obtain
    $$
    \frac{f(0)}{2}|x_n|^{2-\alpha} \leq [w]_{C_x^{0, \alpha}(B_r \times \{-r^2\})} \leq Cr^{3 - \delta -\alpha}.
    $$
    Taking $\alpha = 1 - \beta$ with $\beta > 0$ we have for any $\eps > 0$
    $$
    |x_n| \leq Cr^\frac{2 - \delta +\beta}{1 + \beta} \leq Cr^{2 - \eps}
    $$
    holds inside of $\{u = 0 \} \cap \left( B_r \times \{-r^2 \}\right)$ taking $\delta$ and $\beta$ small enough. Since the region $\{u = 0 \}$ shrinks with time the claim follows. \end{proof}
\smallskip
\begin{lemma} 
\label{lem:aux_bound_cleaning_parabolic_top_stratum}
Let $u: B_1 \times (-1, 1) \rightarrow [0, \infty)$ be a bounded solution of \eqref{eq:stefan_problem}, $(0, 0) \in \Sigma_{n-1}$. Assume without loss of generality $p_{2} = \frac{f(0)}{2} x_n^2$. Then for any $\eps >0$, there exists some positive constants $\bar{c}$ and $r_\eps$ such that  
    $$
    \fint_{B_r \cap \{|x_n| \geq \frac{r}{10} \}} \partial_t u(\cdot, -r^2) dx \geq \bar{c} r \quad \text{for all } r \in (0, r_\eps).
    $$
\end{lemma}
\begin{proof}
    Take $\rho > 0$ small enough, by \cref{lem:equiv_to_FRS6.6} applied to the rescaled function $\rho^{-2}u(\rho \cdot, \rho^2\cdot)$ and taking as radius the parabolic combination $r^2 := |x'|^2 - t$ we have for any $\eps > 0$
    $$
    \{u(\cdot, \cdot) = 0\} \cap B_\rho \subset \{|x_n| \leq C(|x'|^2 -t)^{1 - \frac{\eps}{2}}\} \quad \text{for all } t \in [-\rho^{2}, 0].
    $$
    Since the domain $\{|x_n| \leq C(|x'|^{2} -t)^{1 - \frac{\eps}{2}}\}$ is a parabolic $C_p^{1, \alpha}$ domain with $\alpha \leq 1 - \eps$ the Hopf lemma holds \cite[Theorem 3.1]{AN19} (see also \cite{Tor24} for a similar result in the more general domains $C^{1, \text{Dini}}$ and \cite{Lie96} Ch. XI \& X for general theory). We can apply Hopf lemma to the function $\partial_t u$ and obtain
    $$
    \partial_tu(x, t) \geq c \ \text{dist}_\text{par}(x, \{u(\cdot, \cdot) = 0\}) \geq cr,
    $$
    and the claim follows.
    
    Note that it is important to use the parabolic distance since $\{|x_n| \leq C(|x'|^{2} -t)^{1 - \frac{\eps}{2}}\}$ is not $C^{1, \alpha}$ with the Euclidean distance.
\end{proof}

\begin{proof}[Proof of \cref{lem:cuad_cleaning_S3_n-1}]
    Is analogous to \cref{lem:cuad_cleaning_Sm} using the bound from \cref{lem:aux_bound_cleaning_parabolic_top_stratum} and \cref{lem:quadratic_cleaning} with $\omega(r) = r^{3 - \eps}$ and $\beta = 1$. Note that the bound over $H(r, u - p_2)^{1/2}$ follows from \cref{prop:lamb=gam} and \cref{cor:H_and_D_bound} taking $\gamma = 3 - \eps$. Note that \cref{cor:H_and_D_bound} can be extended for any $\gamma \in (2, 3)$ thanks to the iterative argument from \cref{prop:lamb=gam}.
\end{proof}

From \cref{lem:cuad_cleaning_Sm} and \cref{lem:cuad_cleaning_S3_n-1} we deduce the following.
\smallskip
\begin{corollary}
    \label{cor:quadratic_cleaning}
    For any $(x_0, t_0) \in \Sigma$ and $\eps > 0$, there exists $r_0 \in (0, 1)$, depending on $x_0, t_0$ and $\eps$ such that
    $$
    \{u = 0\} \cap \left(B_{r}(x_0) \times [t_0 + r^{2 - \eps}, 1) \right) = \emptyset \quad \text{for all } r \in (0, r_0).
    $$
\end{corollary}
To conclude the proof, we need a technical lemma from Geometric Measure Theory.
\smallskip
\begin{lemma}[{\cite[Lemma 8.9]{FRS24}}]
    \label{lem:dimension_reduction_lemma}
    Let $E \subset \R^n \times (-1, 1)$ with
    $$
    \text{dim}_\mathcal{H} \left( \pi_x(E)\right) \leq \beta.
    $$
    Assume that for any $\eps > 0$ and $(x_0, t_0) \in E$ there exists $\rho = \rho(\eps, x_0, t_0) > 0$ such that
    $$
    \left\{ (x, t) \in B_\rho(x_0) \times (-1, 1) : t - t_0 > |x - x_0|^{2-\eps}\right\} \cap E = \emptyset.
    $$
    Then $\text{dim}_\text{par}(E) \leq \beta$.
\end{lemma}

\begin{proof}[Proof of \cref{thm:main_theorem}]
The result follows from \cref{thm:cover_of_singular_projection}, \cref{cor:quadratic_cleaning} and \cref{lem:dimension_reduction_lemma}.
\end{proof}

\end{document}